\documentclass[10pt]{article}

  \usepackage[all]{xy}
 \usepackage{graphicx}
 \usepackage{units}
 
 \usepackage{enumerate}
 
\usepackage{amssymb, amsmath, pict2e}
\usepackage{pdfsync}

\usepackage{rotating}

\usepackage{multirow}

  \usepackage[normalem]{ulem}
  \usepackage{cancel}

\usepackage{color}
\usepackage[usenames,dvipsnames,svgnames,table]{xcolor}

\title{Unit Mixed Interval Graphs}

\author{Alan Shuchat \\  Department of Mathematics \\ Wellesley College \\
Wellesley, MA 02481 \ \ USA \\ \and
Randy Shull \\ Department of
Computer Science        \\ Wellesley College
\\ Wellesley, MA 02481 \ \ USA \\
\and Ann N.\ Trenk \\ Department of Mathematics \\ Wellesley College \\
Wellesley, MA 02481 \ \ USA\ \
\and Lee C. West \\Department of Mathematics \\Wellesley College\\ Wellesley, MA 02481 \ \ USA
}

\date{May 7, 2014}

\newtheorem{theorem}{Theorem}
\newtheorem{defn}[theorem]{Definition}

\newtheorem{remark}[theorem]{Remark}

\newtheorem{lemma}[theorem]{Lemma}

\newtheorem{prop}[theorem]{Proposition}
\newtheorem{corollary}[theorem]{Corollary}

\newtheorem{hyp}[theorem]{Hypothesis}
\newtheorem{alg}[theorem]{Algorithm}

\newcommand{\qed}{\mbox{$\Box$}}

\def\reals{{\mathbb R}}

\newcommand{\ahat}{{\hat{a}}}
\newcommand{\bhat}{{\hat{b}}}
\newcommand{\chat}{{\hat{c}}}

\newcommand{\bolda}{{\bf{a}}}
\newcommand{\boldb}{{\bf{b}}}

\begin{document}

  \maketitle

\bibliographystyle{plain}

\begin{center} {\sl ABSTRACT} \end{center}

\begin{quotation}

 In this paper we extend the work of Rautenbach and Szwarcfiter \cite{RaSz13} by giving a structural characterization of graphs that can be represented by the intersection of unit intervals that may or may not contain their endpoints. A characterization was proved independently by Joos in \cite{Jo13}, however our approach provides an algorithm that produces such a representation,  as well as a forbidden graph characterization.

\end{quotation}

\section{Introduction}

Interval graphs are important because they can be used to model problems in the real world,   have elegant characterization theorems, and efficient recognition algorithms.  In addition, some graph problems that are known to be difficult in general, such as finding the chromatic number, can be solved efficiently when restricted to the class of interval graphs (see for example, \cite{Go80,GoTr04}).

A graph $G$ is an \emph{interval graph} if each vertex can be assigned an interval on the real line so that two vertices are adjacent in $G$ precisely when the intervals intersect.  In some situations, all the intervals will have the same length and such a graph is called a \emph{unit interval graph}.   The   unit interval graphs are characterized  in \cite{Ro69} as those interval graphs with no induced claw $K_{1,3}$.

Many papers about interval graphs do not specify whether the assigned intervals are open or closed, and indeed the class of interval graphs is the same whether open or closed intervals are used.  The same is true for the class of unit interval graphs.    Rautenbach and Szwarcfiter \cite{RaSz13} consider the class of graphs that arise when both open and closed intervals are permitted in the same representation.  The class of interval graphs remains unchanged even when both open and closed intervals are allowed.  However,  the class of unit interval graphs is enlarged.  In particular, the claw $K_{1,3}$ can be represented (uniquely) using one open interval and three closed intervals.  Rautenbach and Szwarcfiter  \cite{RaSz13} use the notation $U^\pm$ to designate this class and characterize the graphs in $U^\pm$  as those interval graphs that do not contain any of seven  forbidden graphs.

This result is  established   by considering twin-free graphs, that is, graphs in which no two vertices have the same closed neighborhood.  For twin-free graphs, the characterization of $U^\pm$ requires only four forbidden graphs. Limiting attention to twin-free graphs is not a substantive restriction:  if a graph contains twins  we can remove duplicates, apply the structural characterization to the resulting twin-free graph, and then restore each duplicate, giving it an interval identical to that of its twins.  

Dourado et al.~\cite{DoLe} generalize the results of \cite{RaSz13} to the class of \emph{mixed interval graphs}, that is to graphs that can be represented using intervals that are open, closed, or half-open.  They further pose a conjecture characterizing the class of mixed unit interval graphs in terms of an infinite family of forbidden induced subgraphs, proving the conjecture for the special case when the graph is diamond-free.  Le and Rautenbach \cite{LeRa} characterize graphs that have  mixed unit interval representations in which all intervals have integer endpoints and provide a quadratic-time algorithm that decides if a given interval graph admits such a representation.   

The present paper addresses the conjecture given in \cite{DoLe}.    The main result is to give a structural characterization of the class of twin-free unit mixed interval graphs.  Our characterization includes a set $\cal F$ consisting of five individual forbidden graphs and  five infinite forbidden families.  A characterization was proved independently by Joos in \cite{Jo13}.  In addition to characterizing the class, our approach provides a quadratic-time  algorithm that takes a twin-free, $\cal F$-free interval graph and produces a mixed interval representation of it, where the only proper inclusions are between intervals with the same endpoints.

\section{Preliminaries}
We denote the left and right endpoints of a real interval $I(v)$ by $L(v)$ and $R(v)$, respectively.  
We say an interval is \emph{open on the left} if it does not contain its left endpoint and \emph{closed on the left} if it does.   Open and closed on the right are defined similarly.
In the introduction, we gave the usual definition of interval graph; we now give a more nuanced definition that allows us to specify which types of intervals are permissible.  

 \begin{defn}  {\rm Let $\cal R$ be a set of real intervals.  An {$\cal R$}-\emph{intersection representation} of a graph $G $ is an assignment ${\cal I}: x \rightarrow I(x)$ of an interval $I(x) \in \cal R$ to each  $x \in V(G)$ so that  $xy \in E(G)$ if and only if $I(x) \cap I(y) \neq \emptyset$.}

 \end{defn}
 
 Throughout this paper, we will denote the classes of closed, open, and half-open intervals by  ${\cal A} = \{ [x,y]: x,y \in \reals \}$,  \  ${\cal B} = \{ (x,y): x,y \in \reals \}$, \   ${\cal C} = \{ (x,y]: x,y \in \reals \}$, \   ${\cal D} = \{ [x,y): x,y \in \reals \}$.
 
 \begin{defn} {\rm
  A graph $G$ is an \emph{interval graph} if it has an $\cal A$-intersection representation.  If in addition, all intervals in the representation have the same length, then $G$ is a \emph{unit interval graph}.  We call this a (\emph{unit}) \emph{closed interval representation} of $G$.}
   \end{defn}

 \begin{defn} {\rm
  A graph $G$ is a \emph{mixed interval graph} if it has an ${\cal R}$-intersection representation, where $\cal R = {\cal A} \cup {\cal B} \cup {\cal C} \cup {\cal D}$.  If in addition, all intervals in the representation have the same length, then $G$ is a (\emph{unit}) \emph{mixed  interval graph}. We call this a (\emph{unit}) \emph{mixed interval representation}.  }
 \label{R-defn}
 \end{defn}
  The proof of the following proposition is part of a similar result from \cite{RaSz13} and  a sketch is included here for completeness.
 \begin{prop}
 A graph   is a mixed interval graph  if and only if it is an interval graph.
  \label{mixed-int-prop}
  \end{prop}
    \begin{proof}By definition, interval graphs are mixed interval graphs.  For the converse,
    let $G= (V,E)$ have a mixed interval representation $\cal I$.  For each edge $uv \in E$, pick a point $x_{uv}$ in the set $I(u) \cap I(v)$.  For each vertex $v \in V$, define $I'(v) = [\min\{x_{uv}: uv \in E\}, \max\{x_{uv}: uv \in E\}]. $  One can check that the intervals $I'(v)$ give a closed interval representation of $G$, so $G$ is an interval graph.  
        \qed
    \end{proof}

   An interval graph is \emph{proper} if it has an ${\cal A}$-intersection representation in which no interval is  properly contained in  another.  By definition, the class of unit interval graphs is contained in the class of proper interval graphs, and in fact, the classes are equal \cite{BoWe99}.  When both open and closed intervals are permitted in a representation, we must refine the notion of proper in order to maintain  the inclusion of the unit class in the proper class.
  
 \begin{defn} {\rm  An interval $I(u)$ is \emph{strictly contained} in an interval $I(v)$ if $I(u) \subset I(v)$ and they do not have identical endpoints.
 An ${\cal R}$-intersection representation is  \emph{strict} if no interval in ${\cal R}$ is strictly contained in another, i.e., if
      the only proper inclusions allowed are between intervals with the same endpoints. }
  \end{defn}
         
  \begin{figure}
   
  \centering
{\includegraphics[height=1in]{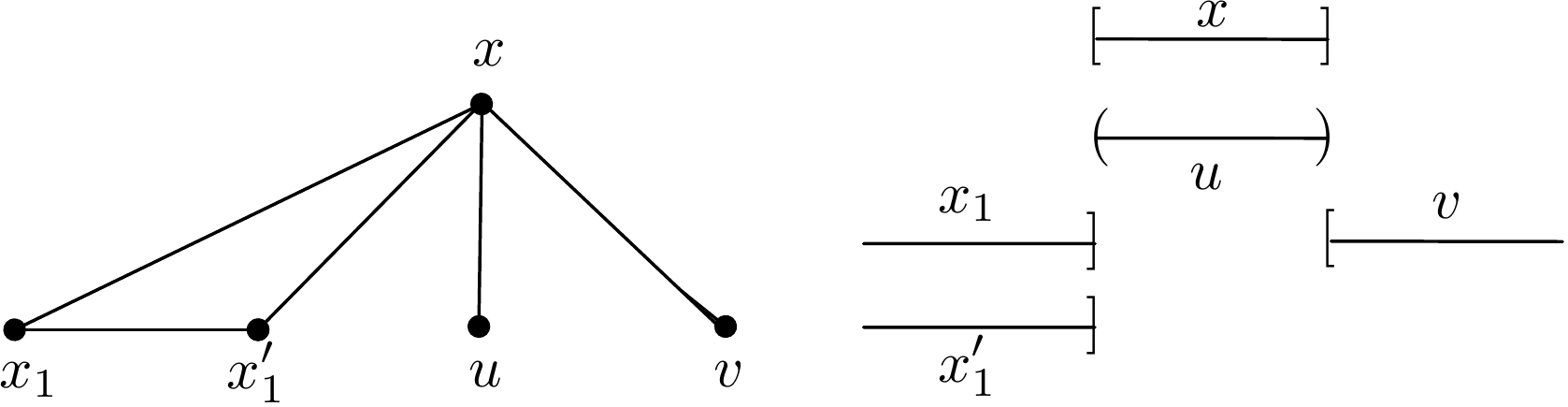}}

   \caption{$K_{1,4}^\ast $ is a unit mixed interval graph.}
  \label{fig-K14-star}
    \end{figure}

	  \begin{figure}   
	 
\centering
{\includegraphics[height=1.26in]{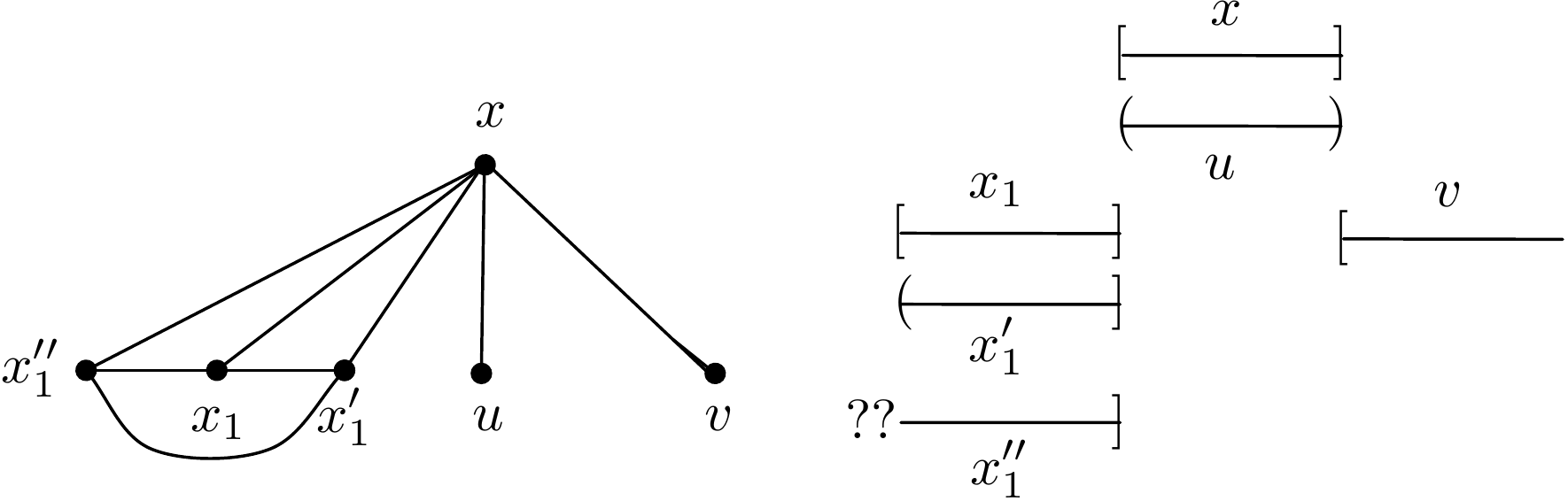}}
\caption{The forbidden graph $B$, which cannot be induced in a twin-free unit mixed interval graph.}  
  \label{fig-B}
   \end{figure}

  \subsection{The forbidden graphs}

  In this section we describe the set ${\cal F}$ of graphs that are forbidden for unit mixed interval graphs.     The claw $K_{1,3}$   is induced by the vertices $x, x_1, u,v$ in the graph $K_{1,4}^*$ shown in Figure~\ref{fig-K14-star}.  The unit mixed interval representation of $K_{1,3}$ shown in Figure~\ref{fig-K14-star} is determined up to relabeling  the vertices, reflecting the entire representation about a vertical line, and specifying whether the outer endpoints of $x_1$ and $v$ are open or closed (which can be done arbitrarily in this case).  In this  and other figures, we shorten the notation by labeling an interval as $x$ instead of $I(x)$ when the meaning is clear.

  With the addition of vertex $x_1'$, Figure~\ref{fig-K14-star}   shows a representation of $K_{1,4}^\ast$ as a unit mixed interval graph.  This representation  can be completed by making the three unspecified endpoints open or closed.  If   $K_{1,4}^\ast$
 is induced in a  twin-free, unit mixed interval graph $G$, then $x_1$ and $x_1'$ cannot be twins in $G$ so one interval must have a closed left endpoint and the other an open one.   As a result,  the graph $B$ shown in Figure~\ref{fig-B} cannot 
 be induced in a twin-free unit mixed interval graph, and indeed,  graph $B$ is  one of the forbidden graphs shown in Figure~\ref{fig-five-forb}.  Similar arguments show that the other graphs in Figure~\ref{fig-five-forb} cannot be induced in $G$. This gives us the following result.
  
  \begin{lemma}
  \label{finite-forb}
  If $G$ is a twin-free unit mixed interval graph then $G$ does not contain any of the five graphs in Figure~\ref{fig-five-forb} as an induced subgraph.
  \end{lemma}
\begin{figure}
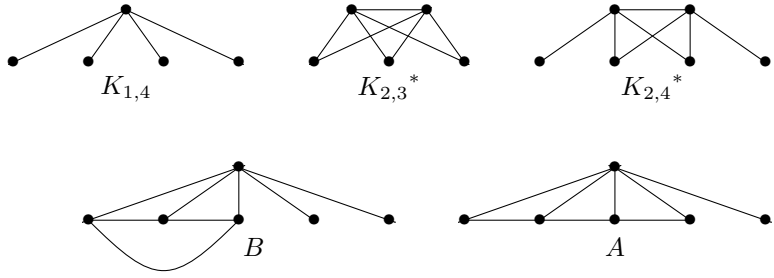

     \centering
  \begin{displaymath}
    \xygraph{
  !{<0cm,0cm>;<1cm,0cm>:<0cm, .7cm>::}
  !{(1.5,21)}*-{\bullet}="p1"
  !{(0,20)}*-{\bullet}="p2" 
  !{(1,20)}*-{\bullet}="p3" 
  !{(2,20)}*-{\bullet}="p4"
  !{(3,20)}*-{\bullet}="p5"
        !{(1.5,19.5)}*-{K_{1,4}}
  !{(4.5,21)}*-{\bullet}="p6"
  !{(4,20)}*-{\bullet}="p7"
  !{(5.5,21)}*-{\bullet}="p8"
  !{(5,20)}*-{\bullet}="p9"
  !{(6,20)}*-{\bullet}="p10"
          !{(5,19.5)}*-{{K_{2,3}}^*}
  !{(7,20)}*-{\bullet}="p11"
  !{(8,21)}*-{\bullet}="p12"
  !{(8,20)}*-{\bullet}="p13"
  !{(9,20)}*-{\bullet}="p14"
  !{(9,21)}*-{\bullet}="p15"
  !{(10,20)}*-{\bullet}="p16"
            !{(8.5,19.5)}*-{{K_{2,4}}^*}
  !{(1,17)}*-{\bullet}="p17"
  !{(3, 18)}*-{\bullet}="p18"
  !{(2,17)}*-{\bullet}="p19"
  !{(3,17)}*-{\bullet}="p20"  
  !{(4,17)}*-{\bullet}="p21"
  !{(5,17)}*-{\bullet}="p22"
            !{(3.2,16.5)}*-{B}
  !{(6,17)}*-{\bullet}="p30"  
  !{(7,17)}*-{\bullet}="p31"
  !{(8,17)}*-{\bullet}="p32"
  !{(9,17)}*-{\bullet}="p33"
  !{(10,17)}*-{\bullet}="p34"
  !{(8,18)}*-{\bullet}="p35"
                !{(8,16.5)}*-{A}
  "p1"-"p2"
  "p1"-"p3"
  "p1"-"p4"
  "p1"-"p5"
  "p6"-"p7"
  "p6"-"p8"
  "p6"-"p9"
  "p6"-"p10"
  "p8"-"p9"
  "p8"-"p7"
  "p8"-"p10"
  "p11"-"p12"
  "p12"-"p13"
  "p12"-"p14"
  "p12"-"p15"
  "p13"-"p15"
  "p14"-"p15"
  "p15"-"p16"
  "p17"-"p18"
  "p17"-"p19"
  "p17"-@(dr,dl)"p20"
  "p19"-"p18"
  "p20"-"p19"
  "p20"-"p18"
  "p18"-"p21"
  "p18"-"p22"
  "p30"-"p31"
  "p30"-"p35"
  "p31"-"p32"
  "p31"-"p35"
  "p32"-"p33"
  "p32"-"p35"
  "p33"-"p35"
  "p34"-"p35"
                }
                  \end{displaymath}
 \caption{Five forbidden graphs in the set {$\cal{F}$} of all forbidden graphs.}
 \label{fig-five-forb}
\end{figure}

   Next we consider the infinite family of graphs $H_k, k \ge 1$, and show that like $K_{1,4}^*$,  each of the graphs in this family is a unit mixed interval graph whose representation is almost completely determined.  The base case of this family is $H_1=K_{1,4}^*$. As an illustration, Figure~\ref{fig-H3} shows a unit mixed interval representation of $H_3$. It is completely determined up to permuting vertex labels, reflecting the entire representation about a vertical line, and specifying whether the right endpoint of $v$ is open or closed.  Lemma~\ref{Hk-unique-lem} shows that the same is true for each graph $H_k$.
 
 \begin{lemma}
 \label{Hk-unique-lem}
 For each $k \ge 1$, the graph $H_k$ shown in Figure~\ref{fig-Hk} is a unit mixed interval graph.  If it is induced in a twin-free graph $G$ then the representation of $H_k$   is completely determined up to permuting vertex labels, reflecting the entire representation about a vertical line and specifying whether the unspecified endpoint of $v$ is open or closed.  Furthermore, for $0 \le j \le k$, the intervals assigned to $a_j$ and $c_j$ have the same endpoints.
 \end{lemma}
 
 \begin{proof}
A unit mixed interval representation of $H_k$ can be constructed as in Figure~\ref{fig-H3}. To prove uniqueness, we proceed by induction on $k$.  Since $H_1 = K_{1,4}^*$, the result for the base case follows from the discussion preceding Lemma~\ref{finite-forb}.  
 
 Assume our result holds for $H_k$ and consider $H_{k+1}$.  Since $H_k$ is induced in $H_{k+1}$, the representation $\cal I$ of $H_k$ is completely determined up to the conditions in the conclusion.  Without loss of generality, we may assume $I(a_j)$ is closed, $I(c_j)$ is open on the left, and the right endpoint of $I(v)$ is unspecified, as illustrated for $H_3$ in Figure~\ref{fig-H3}.  
 
 The intervals assigned to $a_{k+1}$ and $c_{k+1}$ must intersect $I(a_k)$ but not $I(c_k)$ and thus must each have a closed right endpoint at $L(a_k)$.   The representation is unit, so $L(a_{k+1}) = L(c_{k+1})$.  Since $a_{k+1}$ and $c_{k+1}$ cannot be twins in $G$, one interval must have an open left endpoint and the other a closed one.  This completely determines the representation of $H_{k+1}$ up to the conditions in the conclusion of the lemma.
 \qed
 \end{proof}
 
 \medskip
 
 Figure~\ref{fig-Hk-fam}  shows the first four infinite forbidden families in the set  ${\cal F}$.  In each case, the graph $H_k$ is attached where the arrow indicates, that is, the triangle of ``tail'' vertices $a_{k-1}, a_k, c_k$ in $H_k$ in Figure~\ref{fig-Hk} is superimposed on the same triangle in each graph in Figure~\ref{fig-Hk-fam}.
 
The last  infinite family in the set ${\cal F}$ arises from the interaction of an $H_k$ graph (with tail vertices $a_k$ and $c_k$) with an $H_n$ graph (with tail vertices $a_n'$ and $d_n'$) as shown in Figure~\ref{fig-Hkn-fam}.

  \begin{prop}
  \label{forb-graphs-prop}
  Let   $\cal F$   be composed of the five individual graphs in Figure~\ref{fig-five-forb}  and the five infinite families shown in Figures~\ref{fig-Hk-fam} and \ref{fig-Hkn-fam}.
If $G$ is a twin-free unit mixed interval graph, then $G$ does not contain any of the graphs in $\cal F$ as an induced subgraph.
  \end{prop}
  
  \begin{proof}
  The result follows for the five graphs in Figure~\ref{fig-five-forb} by Lemma~\ref{finite-forb}.  Next consider the graphs in Families 1 -- 4.  In each case, begin with a representation of $H_k$, which is almost completely determined by Lemma~\ref{Hk-unique-lem}.  Without loss of generality, we may assume the intervals in the representation proceed from right to left, with tail vertices $a_k$ assigned to a closed interval and $c_k$  to one that is open on the left, as illustrated in Figure~\ref{fig-H3} for the case $k=3$.   Now consider a graph in Family 1.    It is impossible to assign unit intervals to vertices $x,y, z$,  each of which intersects $I(a_k)$ but not $I(c_k)$, without using identical intervals and thus creating twins.  This is the same argument we used for graph $B$ of Figure~\ref{fig-B}.  The arguments for Families 2-4 are similar.

  Finally, let $H$ be a unit mixed interval graph in Family 5, where $H$ consists of  $H_k$ with tail vertices $a_k, c_k$ and  $H_n$ with tail vertices $a_n', d_n'$.  Fix a unit mixed interval representation of $H$. 
   By Lemma~\ref{Hk-unique-lem}, the representations of $H_k$  and $H_n$ are almost completely determined and the intervals assigned to $a_k$ and $c_k$ have the same endpoints and the intervals assigned to $a_n'$ and $d_n'$ have the same endpoints.  Now it is impossible to have all pairs of vertices in the set $\{a_k,c_k, a_n', d_n'\}$ adjacent in $H$ except for $c_k, d_n'$. 
Thus $G$ cannot contain any graph in $\cal F$. \qed
  \end{proof}

\begin{figure}

\begin{center}

{\includegraphics[height=1.77in]{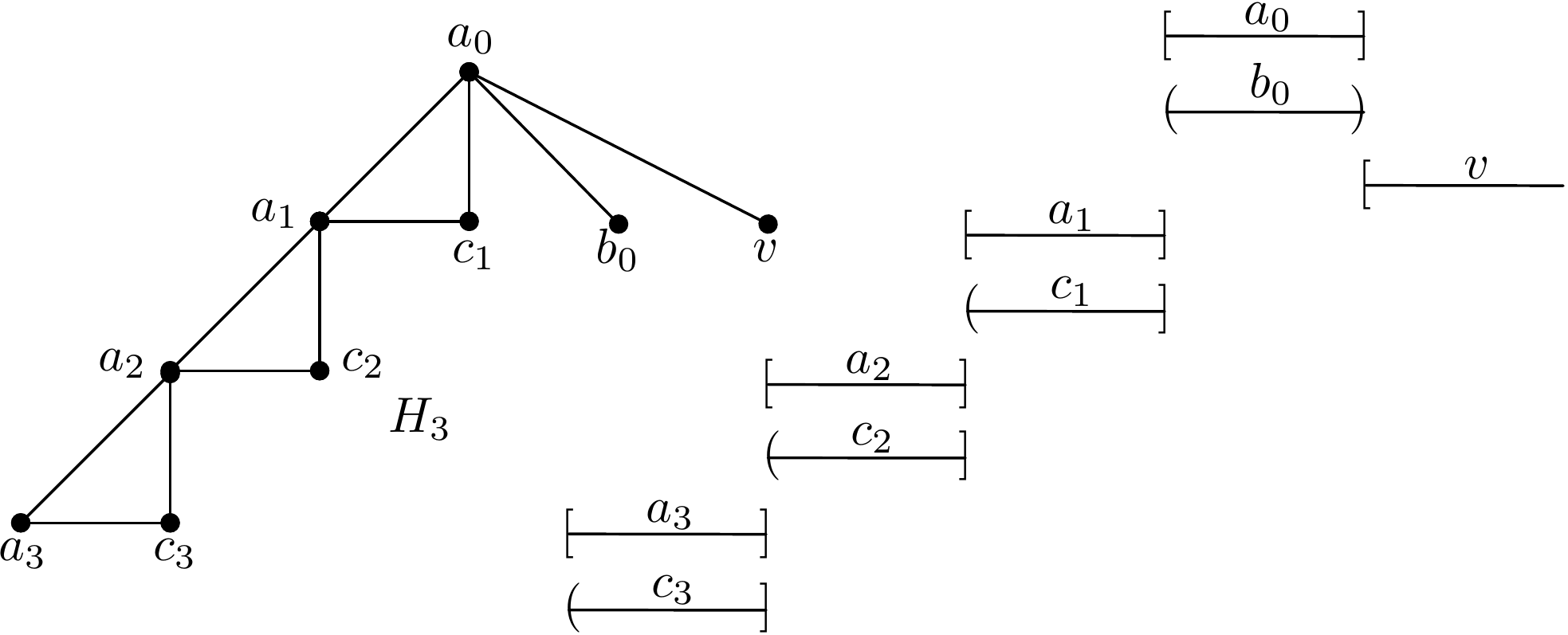}}

\end{center}
\caption{The almost unique representation of $H_3$. }
 \label{fig-H3}
\end{figure}

\begin{figure}

  \begin{center}

{\includegraphics[height=1.96in]{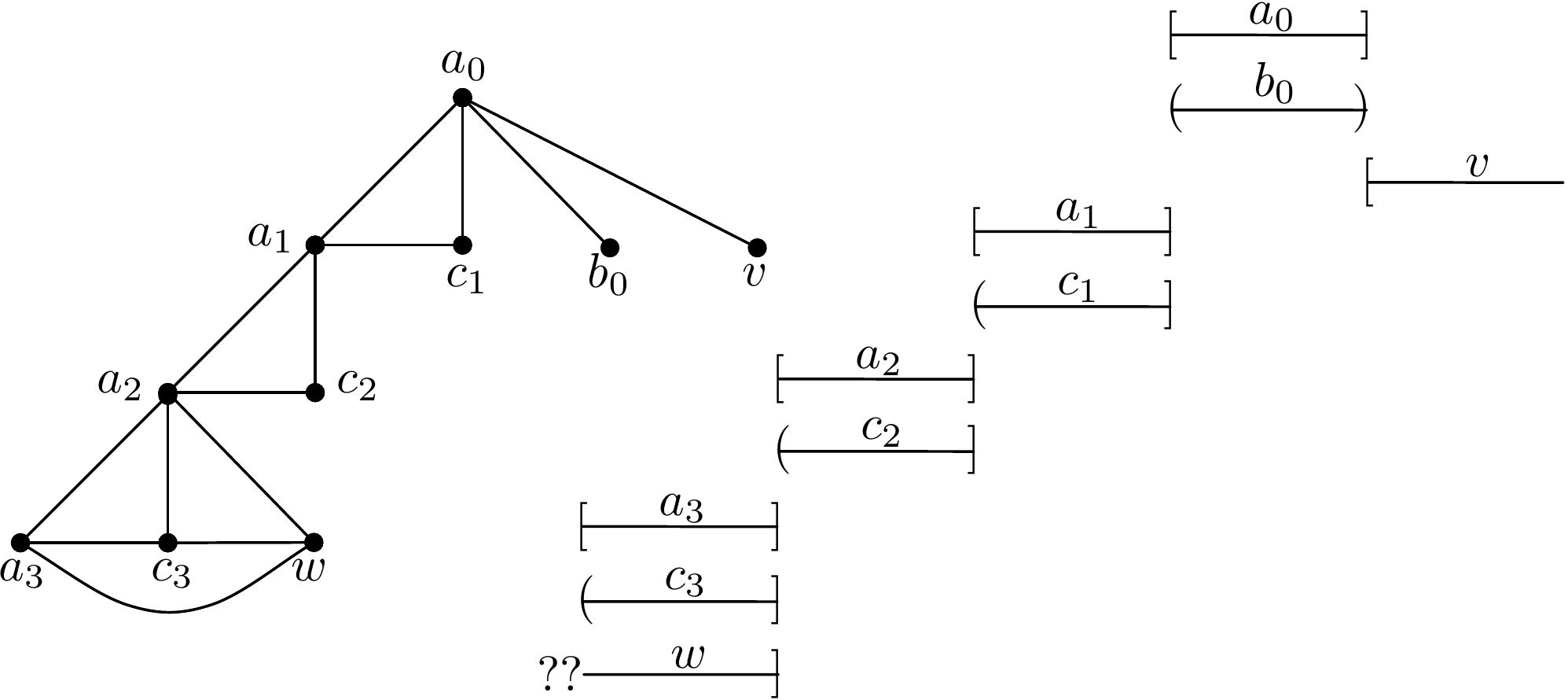}}

  \end{center}
  \caption{The graph from forbidden Family 1 with $k = 2$. }
   \label{fig-fam1}
  \end{figure}
    
 
\begin{figure}

 \centering
  \begin{displaymath}
    \xygraph{
  !{<0cm,0cm>;<1cm,0cm>:<0cm, .7cm>::}
  !{(5,20)}*-{\bullet}="h1"
      !{(5,20.3)}*-{a_0}
  !{(3.5,19)}*-{\bullet}="h2"
      !{(3.2,19.1)}*-{a_1}
  !{(4.5,19)}*-{\bullet}="h3"
        !{(4.6,18.7)}*-{c_1}
  !{(5.5,19)}*-{\bullet}="h4"
        !{(5.5,18.7)}*-{b_0}
  !{(6.5,19)}*-{\bullet}="h5"
          !{(6.5,18.7)}*-{v}
  !{(2.5,18)}*-{\bullet}="h6" 
        !{(2.2,18)}*-{a_2} 
  !{(3.5,18)}*-{\bullet}="h7"
          !{(3.5,17.7)}*-{c_2}
  !{(1.5,17)}*-{\bullet}="h11"
            !{(1,17)}*-{a_{k-1}}
  !{(0.5,16)}*-{\bullet}="h12"
  !{(1.5,16)}*-{\bullet}="h13"
          !{(1.5,15.7)}*-{c_k}
           !{(.5,15.7)}*-{a_k}
  !{(4.5,17.5)}*-{H_k}
  "h11":@{.}"h6"
  "h1"-"h2"
  "h1"-"h3"
  "h1"-"h4"
  "h1"-"h5"
  "h2"-"h3"
  "h2"-"h6"
  "h2"-"h7"
  "h6"-"h7"
  "h11"-"h12"
  "h11"-"h13"
  "h12"-"h13"
  }
  \end{displaymath}

\caption{The infinite family of graphs $H_k, k\ge 1$.}
 \label{fig-Hk}
\end{figure}

\begin{figure}
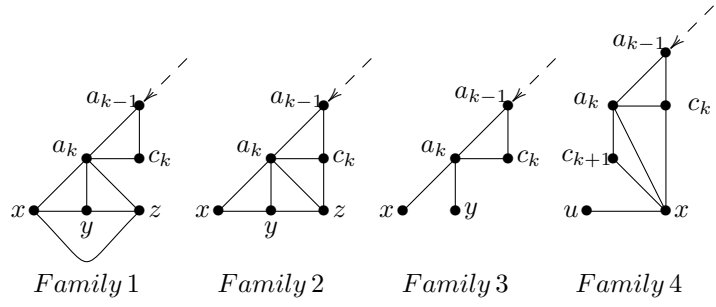


 \centering
  \begin{displaymath}
    \xygraph{
  !{<0cm,0cm>;<.7cm,0cm>:<0cm, .7cm>::}
  !{(4.75,21.25)}*-{}="Hk1"
	!{(4.1, 20.6)}="h14a"
  !{(4,20.5)}*-{\bullet}="h14"
            !{(3.5,20.6)}*-{a_{k-1}}
  !{(3,19.5)}*-{\bullet}="h15"
            !{(2.6,19.7)}*-{a_k}
  !{(4,19.5)}*-{\bullet}="h16"  
            !{(4.4,19.5)}*-{c_k}
  !{(2,18.5)}*-{\bullet}="h17"
            !{(1.7,18.5)}*-{x}
  !{(3,18.5)}*-{\bullet}="h18"
            !{(3,18.15)}*-{y}
  !{(4,18.5)}*-{\bullet}="h19"
            !{(4.3,18.5)}*-{z}
                    !{(3,17.1)}*-{Family \, 1}
   !{(8.25,21.25)}*-{}="Hk2"
	!{(7.6, 20.6)}="h24a"
  !{(5.5,18.5)}*-{\bullet}="h20"
              !{(5.2,18.5)}*-{x}
  !{(6.5,18.5)}*-{\bullet}="h21"
              !{(6.5,18.15)}*-{y}
  !{(7.5,18.5)}*-{\bullet}="h22"
              !{(7.8,18.5)}*-{z}
  !{(7.5,19.5)}*-{\bullet}="h23"
              !{(7.9,19.5)}*-{c_k}
  !{(7.5,20.5)}*-{\bullet}="h24"
              !{(7,20.6)}*-{a_{k-1}}
  !{(6.5,19.5)}*-{\bullet}="h25"
              !{(6.1,19.7)}*-{a_k}
                      !{(6.5,17.1)}*-{Family \, 2}
   !{(11.75,21.25)}*-{}="Hk3"
	!{(11.1, 20.6)}="h26a"
  !{(11,20.5)}*-{\bullet}="h26"  
                !{(10.5,20.7)}*-{a_{k-1}}
  !{(10,19.5)}*-{\bullet}="h27"
              !{(9.6,19.7)}*-{a_k}
  !{(11,19.5)}*-{\bullet}="h28"
                !{(11.4,19.5)}*-{c_k}
  !{(9,18.5)}*-{\bullet}="h29"
              !{(8.7,18.5)}*-{x}
  !{(10,18.5)}*-{\bullet}="h30"
              !{(10.3,18.5)}*-{y}
                      !{(10,17.1)}*-{Family \, 3}
      !{(14.75,22.25)}*-{}="Hk7"
	!{(14.1, 21.6)}="h37a"
  !{(14,21.5)}*-{\bullet}="h37"
    !{(13.5,21.7)}*-{a_{k-1}}
  !{(13,20.5)}*-{\bullet}="h38"
    !{(12.5,20.6)}*-{a_k}
  !{(14,20.5)}*-{\bullet}="h39"
    !{(14.65,20.5)}*-{c_k}
  !{(14,18.5)}*-{\bullet}="h40"
    !{(14.3,18.5)}*-{x}
  !{(13,19.5)}*-{\bullet}="h41"
  !{(12.5,19.5)}*-{c_{k+1}}
  !{(12.5,18.5)}*-{\bullet}="h42"
    !{(12.2,18.5)}*-{u}
        !{(13.3,17.1)}*-{Family \, 4}
     "Hk7":@{-->}"h37a"
    "h37"-"h38"
  "h37"-"h39"
  "h38"-"h39"
  "h38"-"h41"
  "h38"-"h40"
  "h39"-"h40"
  "h40"-"h42"
  "h40"-"h41"
 "Hk1":@{-->}"h14a"
  "h14"-"h15"
  "h14"-"h16"
  "h15"-"h16"
  "h15"-"h17"
  "h15"-"h18"
  "h15"-"h19"
  "h17"-"h18"
  "h18"-"h19"
  "h17"-@(dr,dl)"h19"
   "Hk2":@{-->}"h24a"
  "h20"-"h21"
  "h20"-"h25"
  "h21"-"h25"
  "h21"-"h22"
  "h22"-"h25"
  "h22"-"h23"
  "h23"-"h24"
  "h23"-"h25"
  "h25"-"h24"
  "Hk3":@{-->}"h26a"
  "h26"-"h27"
  "h26"-"h28"
  "h27"-"h28"
  "h29"-"h27"
  "h27"-"h30"
}
  \end{displaymath}

\caption{Four forbidden families formed by attaching $H_k$ as shown.}
 \label{fig-Hk-fam}
\end{figure}

\begin{figure}
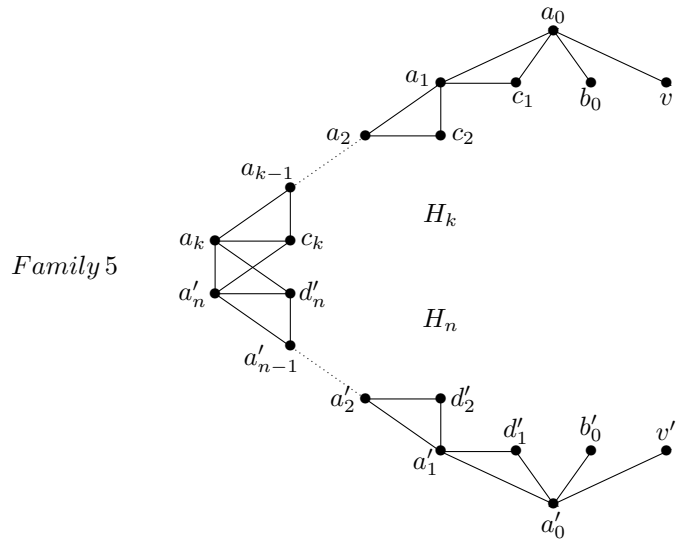


 \centering
  \begin{displaymath}
    \xygraph{
  !{<0cm,0cm>;<1cm,0cm>:<0cm, .7cm>::}
  !{(9,9)}*-{\bullet}="d1"
  !{(9,9.3)}*-{a_0}
  !{(7.5,8)}*-{\bullet}="d2"
  !{(7.2,8.1)}*-{a_1}
  !{(8.5,8)}*-{\bullet}="d3"
  !{(8.6,7.7)}*-{c_1}
  !{(9.5,8)}*-{\bullet}="d4"
  !{(9.5,7.7)}*-{b_0}
  !{(10.5,8)}*-{\bullet}="d5"
  !{(10.5,7.7)}*-{v}
  !{(6.5,7)}*-{\bullet}="d6"
 !{(6.15,7)}*-{a_2}
 !{(7.5,7)}*-{\bullet}="d7"
 !{(7.8,7)}*-{c_2}
  !{(5.5,6)}*-{\bullet}="d11"
  !{(5.2,6.3)}*-{a_{k-1}}
  !{(7.5,5.5)}*-{H_k}
  !{(4.5,5)}*-{\bullet}="d12"
    !{(4.2,5)}*-{a_k}
    !{(2.5, 4.5)}*-{Family \, 5}
  !{(5.5,5)}*-{\bullet}="d13"
  !{(5.8,5)}*-{c_k}
  !{(4.5,4)}*-{\bullet}="d14"
  !{(4.2,4)}*-{a_n'}
 !{(5.5,4)}*-{\bullet}="d15"
 !{(5.8,4)}*-{d_n'}
  !{(5.5,3)}*-{\bullet}="d16"
  !{(5.25,2.75)}*-{a_{n-1}'}
    !{(7.5,3.5)}*-{H_n}
  !{(6.5,2)}*-{\bullet}="d20"
  !{(6.2,2)}*-{a_2'}
  !{(7.5,2)}*-{\bullet}="d21"
  !{(7.8,2)}*-{d_2'}
  !{(7.5,1)}*-{\bullet}="d22"
  !{(7.3,.8)}*-{a_1'}
  !{(8.5,1)}*-{\bullet}="d23"
  !{(8.5,1.4)}*-{d_1'}
  !{(9.5,1)}*-{\bullet}="d24"
  !{(9.5,1.45)}*-{b_0'}
  !{(10.5,1)}*-{\bullet}="d25"
  !{(10.5,1.4)}*-{v'}
 !{(9,0)}*-{\bullet}="d26"
 !{(9,-.4)}*-{a_0'}
   "d1"-"d2"
  "d1"-"d3"
  "d1"-"d4"
  "d1"-"d5"
  "d2"-"d3"
  "d2"-"d6"
  "d2"-"d7"
  "d6":@{.}"d11"
  "d6"-"d7"
  "d11"-"d12"
  "d12"-"d13"
  "d12"-"d14"
  "d13"-"d14"
  "d12"-"d15"
  "d11"-"d13"
  "d14"-"d15"
  "d14"-"d16"
  "d15"-"d16"
  "d20":@{.}"d16"
  "d20"-"d21"
  "d20"-"d22"
  "d21"-"d22"
  "d22"-"d23"
  "d22"-"d26"
  "d23"-"d26"
  "d26"-"d24"
  "d25"-"d26"
}
  \end{displaymath}

\caption{ A fifth forbidden family formed by attaching   $H_k$ to $H_n$.}
 \label{fig-Hkn-fam}
\end{figure}

\section{The Main Theorem}

We are now ready to state our main theorem. We will prove part of it now and the rest will follow from Corollary~\ref{finish-proof}.

\begin{theorem}
 Let $G$ be a twin-free interval graph.  The following are equivalent:

\begin{enumerate}[(1)]  
\item $G$ is a unit mixed interval graph.
  
\item $G$ is a strict mixed interval graph.
  
\item $G$ has no induced graph from the forbidden set $\cal F$ (consisting of the five individual graphs in Figure~\ref{fig-five-forb} and the five infinite families shown in  Figures~\ref{fig-Hk-fam} and \ref{fig-Hkn-fam}).

\end{enumerate}
  
  \label{big-thm}
  \end{theorem}

  We begin by proving that (2) $\Longrightarrow (1) $.  The proof that (1) $\Longrightarrow (3)$ follows from Proposition~\ref{forb-graphs-prop}, and we will defer the proof that (3) $\Longrightarrow (2)$ to Section~\ref{Completing}.

Let $G = (V,E)$ be a strict mixed interval graph and fix an ${\cal R}$-representa\-tion ${\cal I}$ of $G$, where $\cal R = {\cal A} \cup {\cal B} \cup {\cal C} \cup {\cal D}$ as in Definition~\ref{R-defn}.  
Take the closure $\overline{I(v)}= [L(v),R(v)]$ of each interval in this representation and remove duplicates, i.e., say two vertices are equivalent if their intervals have the same closure and take one representative from each equivalence class.  
Let $V^\prime \subset V$ be the set of representative vertices and let $E^\prime \subset E$ be the set of edges induced by $V^\prime$.  Let $H = (V^\prime, E^\prime)$.

  The intervals $\overline{I(v)}$ for $ v \in V^\prime$, determine a proper representation of $H$.
 Apply the Bogart-West procedure in \cite{ BoWe99} to this proper representation to obtain a unit representation ${\cal I'}$ of $H$. For each $v \in V^\prime$, let  $I'(v) = [L'(v),R'(v)]$.  As observed  in \cite{RaSz13}, this construction  satisfies
\begin{equation}
\tag{*}
R(u) = L(v)  \hbox{ if and only if } R'(u) = L'(v), \hbox{for all } u, v\in V'.
\end{equation}

We now reinstate the duplicates we removed by extending ${\cal I'}$ to $G$ as follows. For each $v\in V$ take the equivalent representative vertex $w$ and let $I''(v)$ be $I'(w)$, but with endpoints determined by the original status of $I(v)$ as an element of ${\cal A},  {\cal B},   {\cal C},$ or  ${\cal D}$. We must show the extension ${\cal I''}$ is a representation of $G$.

First suppose  $u,v \in V$ are adjacent in $G$.  
In this case, the intervals  $I(u) $ and $ I(v) $ in the original representation ${\cal I}$ intersect, so their closures intersect. The closures coincide with representative intervals in $\overline{\cal I}$ that intersect (whether or not $u, v \in V^\prime$) and thus to edges in $E^\prime$. 
These edges in turn correspond to intervals in ${\cal I'}$ that intersect. When we restore the intervals deleted from ${\cal I}$,  (*) then implies that $I''(u) $ and $ I''(v) $ intersect.

Now suppose $u$ and $v$ are not adjacent in $G$.    Then either $\overline{I(u)}$ and $\overline{I(v)}$   are disjoint or   intersect in one point.  In  the first case    the representative intervals in ${\cal I'}$ are disjoint and so $I''(u), I''(v)$ are disjoint. In  the second case, without loss of generality we may assume $R(u) = L(v)$, so by (*), $R'(u) = L'(v) $ and then again $I''(u)  \cap I''(v) =\emptyset$, as desired.  \qed


\section{Converting an ${\cal F}$-free interval representation to a strict mixed interval representation}

In the remaining sections we prove
  (3) $\Longrightarrow (2)  $  of Theorem~\ref{big-thm}
   using Algorithm~\ref{sweep}, which turns a closed interval representation of a twin-free, {$\cal F$}-free graph $G$  into a strict mixed interval representation of $G$.  In Section~\ref{initial} we describe certain properties that our initial closed interval representation will satisfy and introduce some needed terminology. We present the algorithm in Section~\ref{algorithm}, and then in Section~\ref{Completing} prove it produces the desired mixed interval representation. 
  
\subsection{Properties of our initial closed representation}
\label{initial}

Our next definition is used throughout the rest of the paper and is illustrated in Figure~\ref{fig-peek},  where $x$ peeks into $ab$ and $ac$ from the left and $y$ peeks into $ab$ from the right.

\begin{defn}{\rm Let $G$ be a graph with a mixed interval representation.
For $x,a,b \in V(G)$, we say $I(x)$ (or $x$) \emph{peeks into $ab$} if $I(x)$ intersects $I(a)$ but not $I(b)$.  Furthermore, it \emph{peeks into $ab$ from the left} if in addition  $ R(x) \le L(b)$ and \emph{peeks into $ab$ from the right} 
 if $R(b) \le L(x).$
 }
\end{defn}

Throughout the rest of this paper, we will be working with a graph and a representation of it  that satisfies the following \emph{initial hypothesis}:

\smallskip
\begin{hyp}
\label{hyp}
 $G$ is a twin-free, $\cal F$-free graph with a closed interval representation having distinct endpoints satisfying the following {\bf inclusion property}: for each proper inclusion $I(u) \subset I(v)$ there exist $x,y \in V(G)$ such that $x$ peeks into $vu$ from the left and $y$ peeks into $vu$ from the right. 
\end{hyp}

\begin{prop}
\label{peeker-prop}
Every twin-free interval graph $G$ has a closed interval representation with distinct endpoints satisfying the inclusion property of Hypothesis~\ref{hyp}.

\end{prop}

\begin{proof}
Given a twin-free interval graph $G$,  use the standard PQ-tree algorithm (see \cite{Go80}) to order the maximal cliques of $G$ as $X_1, X_2, \ldots, X_m$ so that the cliques containing any particular vertex appear consecutively.    Let $v$ be a vertex of $G$ appearing in cliques $X_j, X_{j+1}, \ldots, X_{j+k}$.  The standard method for producing an interval representation of $G$ assigns interval $I(v) = [j,j+k]$ to vertex $v$.  We modify this representation so that vertex $v$ is assigned interval $\hat{I}(v) = [j - \frac{1}{k+3}, j+k + \frac{1}{k+3}]$. Since $G$ is twin-free, the endpoints of the intervals in $\{\hat{I}(v): v \in V(G)\}$ are distinct by construction. 
It is easy to check that $\hat{I}(v)$ properly contains $\hat{I}(u)$ if and only $I(v)$ properly contains $I(u)$ and $I(u), I(v)$ have distinct endpoints.

Now suppose $\hat{I}(v) $ properly contains $\hat{I}(u)$ and let $X_s$ be a maximal clique in $G$ containing $u$ and $v$.  Then by our construction, there exist maximal cliques $X_r$ and $X_t$ each containing $v$ but not $u$ with $r<s$ and $t>s$.    Thus there exist vertices $x \in X_r$ and $y \in X_t$  distinct from $v$ so that $x,y \not\in X_s$.  Then in the original representation $x$ peeks into $vu$ from the left and $y$ peeks into $vu$ from the right.  It follows that  in the modified representation  $x$ peeks into $vu$ from the left and $y$ peeks into $vu$ from the right, as desired. \qed
\end{proof}

\medskip

 \begin{figure}
  \begin{center}
  {\includegraphics[height=0.82in]{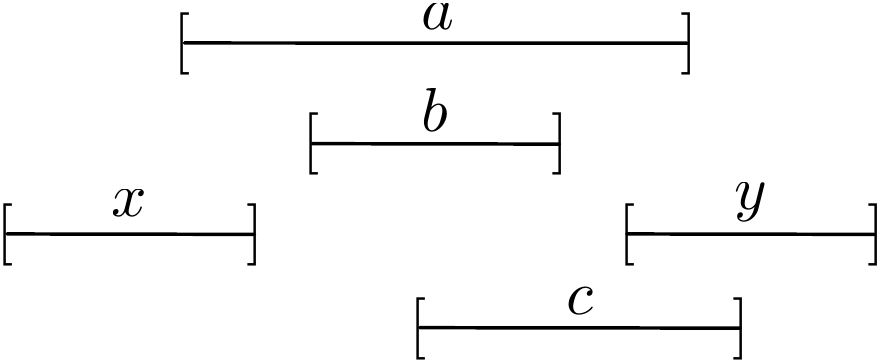}}
  \end{center}
   \caption{$x$ peeks into $ab$ and $ac$ from the left and $y$ peeks into $ab$ from the right.}
   \label{fig-peek}

\end{figure}

In Proposition~\ref{claims},   the intervals in the representation are closed and in their initial positions, before Algorithm~\ref{sweep} is applied.  Later we continue to refer to Proposition~\ref{claims}  for the graph induced by vertices whose  intervals  remain in their initial positions.   
  For closed intervals, strict and proper containment are equivalent, but in mixed interval representations, a containment can be proper without being strict. Proposition \ref{claims} is phrased in terms of strict containment for this reason.    A similar result appears in \cite{RaSz13}.

 \begin{prop}  
\label{claims}  
 If $G$ is a graph  with a closed interval representation satisfying Hypothesis~\ref{hyp} then the following properties hold.
 
 \begin{enumerate}[(1)]

\item No interval can strictly contain two other intervals.

\item No interval  is strictly contained in  two other intervals.

\item If $I(u) \subset I(v)$
then at most two intervals peek into $vu$ from the left and at most two peek into $vu$ from the right.  Furthermore, if two intervals peek in from one side, then neither is contained in the other.

\end{enumerate}

\end{prop}

\begin{proof}
We will prove Proposition~\ref{claims} only for peeking in from the left and for left endpoints since the results for the right are analogous.

\smallskip

\smallskip

{\bf Proof of (1): }  Suppose there exist two distinct vertices 
  $u\ne v$ of $G$ 
whose intervals are strictly contained in $I(w)$.  
 We will show that every possible configuration of these intervals leads to a contradiction. 

First consider the case of strictly nested intervals:  $I(u) \subset I(v) \subset I(w)$.     
By Hypothesis~\ref{hyp} there exist vertices $x,y,t$ for which $x$ peeks into $wv$ from the left, $y$ peeks into $wv$ from the right, and  $t$ peeks into $vu$ from the left.    If $x$ and $t$ are adjacent, then $w,x,t,v,u,y$ induce  the forbidden 
graph $A$ in $G$, a contradiction since $G$ is ${\cal F}$-free.  
(When we describe forbidden graphs in this way, we generally list the vertices from top to bottom and left to right, corresponding to how they are drawn in the accompanying figures.)
Otherwise, if $x$ and $t$ are  not adjacent, then the vertices $w,x,t,u,y$ induce forbidden graph $K_{1,4}$ in $G$,  also a contradiction. 
Thus the representation of $G$ cannot contain three nested intervals.
 
Since they are not nested, we may suppose without loss of generality that $I(u)$ has the leftmost left endpoint and $I(v) $ has the rightmost right endpoint of all the intervals  strictly contained in $I(w)$. 
By Hypothesis~\ref{hyp} there exist vertices $x,y$ for which $x$ peeks into $wu$ on the left and  $y$  peeks into $wv$ on the right.  If $I(u) \cap I(v) = \emptyset$ then the vertices $w,u,v,x,y$ induce the forbidden graph $K_{1,4}$ in $G$, a contradiction.  Hence $I(u) $ and $ I(v)$ intersect.  
 Since $u$ and $v$ are not twins, 
   there must be a vertex $t$
  adjacent to exactly one of them.
   Without loss of generality, assume $I(t)$ intersects $I(u)$ but not $I(v)$.  
      If $x$ and $t$ are not adjacent, then vertices $w,x,t,v,y$ induce $K_{1,4}$ in $G$, a contradiction.  Thus $x$ and $t$ are  adjacent, but then  $w,x,t,u,v,y$ induce  graph $A$ in $G$, which is also a contradiction. 

\smallskip

{\bf Proof of (2):} Suppose there exist $u,v,w$ so that $I(u)$ is strictly contained in both $I(v)$ and $I(w)$.  Without loss of generality 
 we may assume that $L(w) < L(v) < L(u)$.  By (1), no three intervals can be nested so $R(w) < R(v)$.  Then there exist vertices $x,y$ where $x$ peeks into $vu$ from the left and $y$ peeks into $wu$ from the right.  Now the vertices $w,v,x,u,y$ induce the graph $K_{2,3}^\ast$ in $G$, a contradiction.

\smallskip

{\bf Proof of (3):} 
There exist vertices $x,y$ where $x$
peeks into $vu$ from the left and  $y$ peeks into $vu$ from the right.  If two more vertices $x', x''$ peeked into $vu$ from the left then  $v,x,x',x'', u,y$ would induce the forbidden graph $B$ in $G$, a contradiction.  We get a similar contradiction if three vertices peek in from the right.

 To prove the last sentence, suppose $x,x'$    
  peek into $vu$ from the left and  for a contradiction, assume  $I(x') \subset I(x)$.  By Hypothesis~\ref{hyp}, there exist  vertices $y,t$ for which $y$ peeks into $vu$ and $t$ peeks into $xx'$ from the right.  If $I(t)$ is contained in $I(v)$ this contradicts part (1), since  $I(v)$ also contains $I(u)$.  So $I(u)$ is contained in both $I(t)$ and $I(v)$, contradicting (2).  
$\qed$

\end{proof}

\smallskip
 
Next we give a brief overview of
Algorithm~\ref{sweep},  which will be applied to a graph $G$ satisfying Hypothesis~\ref{hyp}.  
Initially all intervals in the representation are closed and their endpoints are colored white.  The algorithm processes intervals  $I(x)$, may modify them by moving one or both endpoints, and may convert an endpoint from closed to open. When an interval is processed, one or both endpoints change from white to red. In Theorem~\ref{alg-correct}, we show that  once an endpoint turns red it is never moved, allowing us to prove that the algorithm terminates.

In Definition~\ref{ab-pair-def},  we define an $ab$-pair. Algorithm~\ref{sweep}  proceeds in complete sweeps, where we begin with an $ab$-pair and process intervals, sweeping left and then right to form a complete sweep. Starting from a particular $ab$-pair $a_0b_0$,   in Definition~\ref{pairs-def} we recursively define vertices $a_j,b_j,c_j,d_j,a_j',b_j'$, $c_j', d_j'$ for $j \ge 0$ (some of which may not exist).  These vertices correspond to the intervals that the algorithm will process during the sweep. After presenting Definition~\ref{pairs-def},  in Lemma~\ref{abc-properties} we show these quantities are well-defined for the initial representation 
and later,  in Theorem~\ref{alg-correct}, show they are well-defined after one or more sweeps of the algorithm.  These definitions  are illustrated in Figure~\ref{pairs} and will  be used in Lemma~\ref{abc-properties} and to formulate   Algorithm~\ref{sweep}.

\begin{defn}
\label{ab-pair-def}
 {\rm Let $G$ be a graph with a mixed interval representation.  We say that vertices $a$ and $b$ form an \emph{$ab$-pair}   if $I(b)$ is strictly contained in $I(a)$. 
 }
\end{defn}

\begin{defn}
\label{pairs-def}
 {\rm   Let $G$ be a graph initially satisfying Hypothesis~\ref{hyp} whose representation may have been modified by   sweeps $S_1, S_2, \ldots, S_i$ of Algorithm~\ref{sweep}.   Let $a_0b_0$ be an $ab$-pair that will initiate a new  sweep $S_{i+1}$ of Algorithm~\ref{sweep}.  
 \begin{enumerate}

\item[(a)]     If there exists a vertex with   the same closed neighborhood as $a_0$ except for vertices peeking into $a_0b_0$ from the left, call such a vertex $c_0$.  Analogously, if there exists a vertex with   the same closed neighborhood as $a_0$ except for vertices peeking into $a_0b_0$ from the right, call such a vertex $d_0$.

\item[(b)]  
Suppose that, moving to the left from $a_0b_0$, we have defined $a_j, b_j, c_j$, $d_j$ for $0 \le j \le k$,   that $a_j$ exists for each $j$, and that at least one of $b_k$ or $c_k$ exists. We next define  $a_{k+1}, b_{k+1}, c_{k+1}, d_{k+1}$ (some of which may not exist).

\begin{itemize}

\item  If only one vertex peeks into $a_kb_k$ or $a_kc_k$ from the left, 
call it $a_{k+1}$;  if there are two, call them $a_{k+1}$ and $ c_{k+1}$, where $R(a_{k+1}) < R(c_{k+1} )$.   When $c_{k+1}$ exists we say $a_{k+1} c_{k+1}$ is a \emph{leftward $ac$-pair} in the representation.

\item  If there exists  a vertex whose representing   interval  is strictly contained in $I(a_{k+1})$, call it $b_{k+1}$.

\item  If there exists a vertex whose closed neighborhood is the same as that of $a_{k+1}$ except that it is not adjacent to $a_{k}$ or $d_{k}$, then call such a vertex $d_{k+1}$.   
 
\end{itemize}

\item[(c)]      When moving to the right from an initial $ab$-pair $a_0b_0$, we typically depict the vertices with primes. That is, we repeat the definition  in part (b) for $k \ge 1$, replacing each occurrence of  $a_k$ by $a_k^\prime$, $b_k$ by $b_k^\prime$, $c_k$ by $d_k^\prime$, $d_k$ by $c_k^\prime$, ``left'' by ``right'', and ``right'' by ``left".
 In particular, this defines  \emph{rightward $ad$-pairs}.
\end{enumerate}
}
\end{defn}

Figure~\ref{pairs} illustrates how the intervals corresponding to this definition are related to each other. The representation may contain additional intervals that are not shown. For simplicity, we often refer to vertices that satisfy the properties of Definition~\ref{pairs-def}, and to their intervals in a given representation, as being of \emph{type} {\bf a},{\bf b},{\bf c}, or {\bf d}. 
Lemma~\ref{abc-properties} 
lists various properties of these vertices. In these proofs and in the rest of this paper we will usually omit ``the tail of'' when saying that vertices induce the tail of a forbidden graph.

\begin{figure}
\begin{center}
 \includegraphics[height=1.25in]{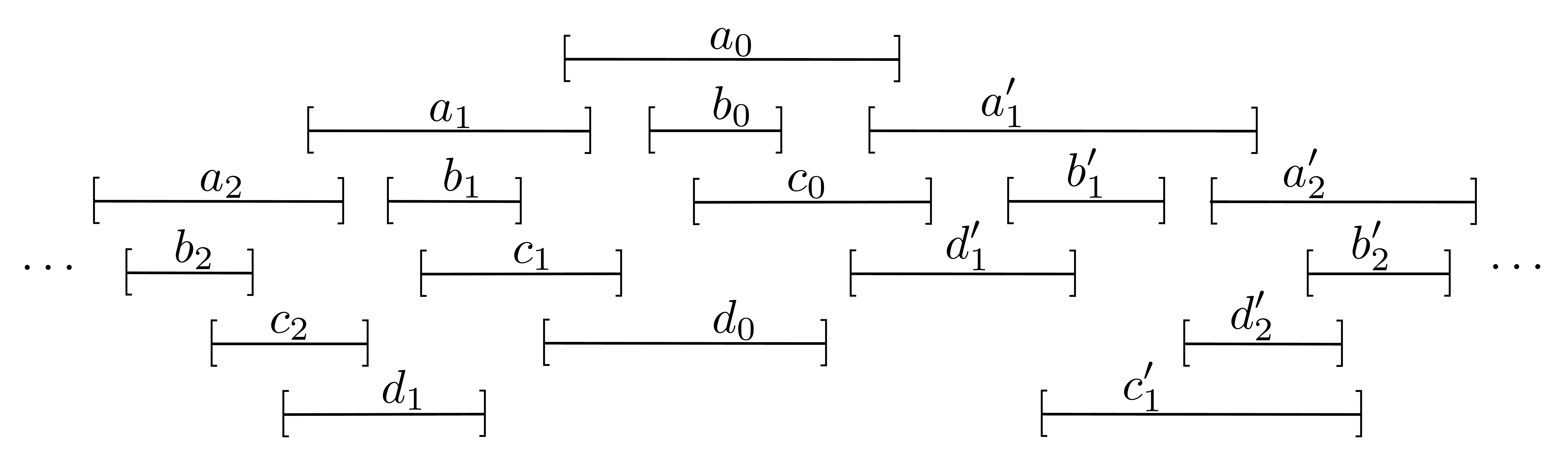}
\end{center}
\caption{Illustrating intervals of types {\bf a}, {\bf b}, {\bf c}, {\bf d}.}
\label{pairs}
 \end{figure}

\medskip

\begin{lemma} \label{abc-properties}
 Let $G$  be  a graph   satisfying Hypothesis~\ref{hyp} whose representation has not been modified by sweeps of Algorithm~\ref{sweep}.
 Fix an $ab$-pair $a_0b_0$. 
 The following properties hold.  
 \begin{enumerate}[(1)]
  \item The quantities in Definition~\ref{pairs-def} are well-defined  if they exist. 
\item  \label{bc-intersect}  If  $a_k b_k$ is an $ab$-pair and $a_kc_k$ is a leftward $ac$-pair, then $I(b_k)$ intersects but is not contained in $I(c_k)$.  
 
\item  \label{ac-ab-peekers} Let $a_kb_k$ be an $ab$-pair and let $a_kc_k$ be a leftward $ac$-pair. Each interval that peeks into $a_kc_k$ from the left also peeks into $a_kb_k$ from the left.

\item  \label{peekers-not-contained} If two intervals peek into an $ab$-pair or a  leftward $ac$-pair   from the left   then neither is contained in the other. 

\item \label{old-AAA}   No vertex in $G$ is adjacent to three  vertices of type $\bf{a}$.

\item  \label{int-c-int-a} If $a_kc_k$ is a leftward  $ac$-pair, then every interval intersecting $I(c_k)$ also intersects $I(a_k)$.

\end{enumerate}

Analogous statements hold for rightward $ad$-pairs.
\end{lemma} 
 
\begin{proof}  We prove the statements for leftward $ac$-pairs.  The proofs of the analogous statements for rightward $ad$-pairs are similar. 

\smallskip

{\bf Proof of (1):   }
     In part (a) there is at most one such $c_0$ because $G$ is twin-free, and similarly there is at most one $d_0$.  In part (b), there are at most two vertices that peek into $a_kb_k$ or $a_kc_k$ from the left for otherwise
  the forbidden graph $B$ or (the tail of) a forbidden graph from Family 1 would be induced in $G$, a contradiction.     If there are two such vertices,   Hypothesis~\ref{hyp} implies that  their endpoints are unequal, so $a_{k+1}$ and $c_{k+1}$ are well-defined.  In addition,      at most
  one   interval  can be strictly contained in $a_{k+1}$ by 
   Proposition~\ref{claims}(1), so $b_{k+1}$ is well-defined.  There is at most one $d_k$ since $G$ is twin-free.  The quantities in part (c) are well-defined analogously.

\smallskip

{\bf Proof of (2):   }
We know  that $I(b_k)$ cannot be contained in $I(c_k)$ by Proposition~\ref{claims}(2).
 Suppose $I(b_k)$ does not intersect $I(c_k)$ and let $y$ be a vertex that peeks into $a_kb_k$ from the left. Then the vertices $a_{k-1},a_k,c_k,y,b_k$ induce in $G$ (the tail of) a graph from Family 3, a contradiction.  \smallskip

\smallskip

{\bf Proof of (3):   }  By Hypothesis~\ref{hyp}, we know there exists a vertex $x$ that peeks into $a_kb_k$ from the left.  
Suppose  there exists a vertex $y$  that peeks into $a_kc_k$ from the left but not into $a_kb_k$. Then $I(y)$ intersects $I(b_k)$, so  $a_{k-1}, a_k,c_k, x, y, b_k$ induce a forbidden graph from Family 2, a contradiction.
\smallskip

{\bf Proof of (4):   }
Suppose $u,v$ peek into  $a_kb_k$ from the left and $I(u) \subset I(v)$.  By Hypothesis~\ref{hyp}, there exists a vertex $y$ peeking into $vu$ from the right.  Then either two intervals are contained in $I(a_k)$ ($I(b_k)$ and $I(y)$) or $I(b_k)$ is contained in two intervals ($I(a_k)$ and $I(y)$).  Both contradict Proposition~\ref{claims}.

Now suppose $k$ is minimum such that there exist $u,v$ peeking into $a_kc_k$ from the left and $I(u) \subset I(v)$.  
If
$b_k$ exists then (3) together with the preceding paragraph gives a contradiction.  So $b_k$ does not exist and, in particular, $k \ge 1$. By the minimality of $k$,  we know $I(a_k) \not\subseteq I(c_k)$ and so $L(a_k) < L(c_k)$.  Again, there exists a vertex $y$ peeking into $vu$ from the right and since $b_k$ does not exist, $I(y) \not\subseteq I(a_k)$.   Thus $I(y)$ intersects $I(c_k)$.

Furthermore, $I(y)$ must intersect
 $I(b_{k-1})$ and/or $I(c_{k-1})$ if they exist, since otherwise the three vertices $a_k,c_k,y$ all peek into $a_{k-1}b_{k-1}$, which induces the forbidden graph $B$ in $G$, and/or peek into $a_{k-1}c_{k-1}$, which induces a forbidden graph from Family 1.  But now if $b_{k-1}$ exists, the vertices $y,a_{k-1},v,c_k,b_{k-1}, z$ induce the forbidden graph $K_{2,4}^*$, where $z$ is a vertex that peeks into $a_{k-1}b_{k-1}$ from the right.  Thus $k \ge 2$ and  $c_{k-1}$ exists.    We must have $R(y) < R(a_{k-1})$ for otherwise $I(y)$ would contain both $I(c_k)$ and $I(a_{k-1})$, violating Proposition~\ref{claims}(1).  However, $R(y) < L(a_{k-2})$, or else the three vertices $y,a_{k-1},c_{k-1}$ all peek into $a_{k-2}b_{k-2}$ or $a_{k-2}c_{k-2}$  and induce $B$ or a forbidden graph from Family 1.  But now $a_{k-2},a_{k-1},c_{k-1},c_k, v, y$ induce in  $G$  a forbidden graph in Family 4, a contradiction.

  \smallskip

{\bf Proof of (5):   }
Assume a vertex $x$ is adjacent to three vertices of type $\bf{a}$. Then its interval $I(x)$ in the given representation must intersect three consecutive intervals of type $\bf{a}$.  If $I(a_0)$ is the middle interval, then $x\ne b_0$ and $I(b_0)$ is contained in both $I(a_0)$ and $I(x)$, contradicting Proposition~\ref{claims}(2).  Thus we may assume without loss of generality that $x$ is adjacent to $a_{k+1}, a_k$, and $a_{k-1}$, where $1 \le k \le n$. Thus $L(x) \le R(a_{k+1})$ and $L(a_{k-1}) \le R(x)$.  If $b_k$ exists then $x \ne b_k$ and $I(b_k)$ is contained in both $I(a_k)$ and $I(x)$, again contradicting Proposition~\ref{claims}(2).   Thus $b_k$ does not exist and $c_k$ must exist.

If $b_{k-1}$ exists then $I(x)$ must intersect $I(b_{k-1})$, or else $x, a_k, c_k$ would peek into  $a_{k-1}b_{k-1}$   from the left, 
contradicting Proposition~\ref{claims}(3).  But now  $x, a_{k-1}, a_{k+1}, c_k, b_{k-1}, a_{k-2}$ induce the forbidden graph $K_{2,4}^\ast$ in $G$, a contradiction (where $a_{-1}$ is interpreted as $a_1'$ if $k=1$).  

Thus $b_{k-1}$ does not exist. Since $b_0$ exists, this implies that $k \ge 2$ and $c_{k-1}$  exists. Hence both $c_k$ and $c_{k-1}$ exist.  If $x$ is not adjacent to $c_{k-1}$, then $a_k, c_k, x$ all peek into $a_{k-1}c_{k-1}$, inducing a forbidden graph from Family 1, a contradiction. So $x$ must be adjacent to $c_{k-1}$, but then $a_{k-2}, a_{k-1}, c_{k-1}, c_k, a_{k+1}, x$ induce  in $G$  a graph in Family 4, 
also a contradiction. 

\smallskip
{\bf Proof of (6):   }
Let $a_kc_k$ be a leftward $ac$-pair.  Suppose there exists a vertex $w$ for which $I(w)$ intersects $I(c_k)$ but not $I(a_k)$.  Thus $L(a_{k-1})  \le R(a_k) < L(w) \le R(c_k)$.
If $b_{k-1}$  exists then either $I(b_{k-1})$ and $I(w)$ are both contained in $I(a_{k-1})$, or $I(b_{k-1})$ is contained in both $I(w)$ and  $I(a_{k-1})$.  Both possibilities contradict Proposition~\ref{claims}.

So $b_{k-1}$  does not exist.
Thus $k \ge 2$ and the vertices $c_{k-1}, a_{k-2}, a_{k-3}$ exist (where $a_1'$ replaces $a_{k-3}$ if $k=2$). Furthermore $I(w) \not\subset I(a_{k-1})$ (or $w$ would be $b_{k-1}$), so $I(w)$ intersects $I(a_{k-2})$.   Let $f_{k-2}$ be $b_{k-2}$ if it exists, and $c_{k-2}$ otherwise.  Now $a_{k-1}$ and $c_{k-1}$ peek into the pair $a_{k-2} f_{k-2}$  so $w$ cannot also peek into this pair,  or this would induce either $B$ or a forbidden graph from Family 1.     Thus $I(w)$ must intersect $I(f_{k-2})$ but  not  $I(a_{k-3})$, by Lemma~\ref{abc-properties}(\ref{old-AAA}).

If $b_{k-2}$ exists, $b_{k-2}=f_{k-2}$, then the vertices $w,a_{k-2},  c_k, c_{k-1}, b_{k-2}, a_{k-3}$ induce the graph $K_{2,4}^\ast$ in $G$, a contradiction.  So $b_{k-2}$  does not exist, $k \ge 3$, and $c_{k-2} = f_{k-2}$ exists. Then the vertices $a_{k-3}, a_{k-2}, c_{k-2}, c_{k-1},$ $c_k, w$ induce  a graph from Family 4 in $G$, a contradiction.   
\qed
\end{proof}

\subsection{The SWEEP algorithm}
 \label{algorithm}
 
In this section we give an overview of Algorithm~\ref{sweep}, and then present the specifics and list the operations  in Tables~\ref{base-step} and \ref{left-k-step}. In Section~\ref{Completing}, we prove the algorithm converts a representation of a graph $G$ satisfying Hypothesis~\ref{hyp} into a strict mixed interval representation of $G$.

As a sweep progresses, starting from an $ab$-pair, vertical zone lines (perpendicular to the intervals) are established in each step of the sweep.  Except for the outermost endpoints of a sweep, endpoints that turn red during a sweep always lie on a zone line of that sweep.   We   maintain the invariant that  after each complete sweep of the algorithm, $xy \in E(G)$ if and only if $I(x) \cap I(y) \neq \emptyset$.

In   the left part of a sweep we identify any vertices $a_k,b_k,c_k,d_k$  that exist, and after the left part is complete their intervals will share the same endpoints. Type ${\bf a}$ intervals will be closed at both endpoints,   type ${\bf b}$ will be open at both endpoints,   type ${\bf c}$ will be open on the left and closed on the right, and    type ${\bf d}$ will be closed on the left and open on the right.  We abbreviate this  by saying that intervals $I(b_k)$,  $I(c_k)$,  and $I(d_k)$ are modified to \emph{match} $I(a_k)$.  Intervals in the right part of a sweep are treated similarly.

\begin{alg}[SWEEP] 
\label{sweep}
{\rm

{\bf Input: } A graph $G$ with a representation satisfying  Hypothesis~\ref{hyp}, with all endpoints  colored white.

\medskip
\noindent
{\bf Output:}  A strict  mixed interval representation of $G$.  

\medskip
\noindent
{\bf One sweep $S_i$ of the algorithm:} Each sweep $S_i$ consists of a base step followed by the steps of the left part of $S_i$ and the right part of $S_i$.

\smallskip
 
{\bf Base step}: Table~\ref{base-step} specifies the operations for the base step of a single sweep.
We identify an $ab$-pair $a_0b_0$, establish zone lines $M_0, M_0^\prime$ at the endpoints of $I(a_0)$, and identify $c_0$ and $d_0$ if they exist.

    The intervals for $b_0,c_0, d_0$ (if they exist) are modified to match $I(a_0)$. 
 We color the endpoints of the intervals for $a_0, b_0, c_0, d_0$ red 

 \smallskip

 \begin{table}
 \begin{center}
  \begin{tabular}{| l | l  | }
  \hline
  Step & Operation  \\ \hline 
[0.1] & \parbox{3.75in}{ Identify an $ab$-pair $a_0b_0$. If none exists, STOP. \\ Otherwise,  establish zone lines $M_0 = L(a_0)$ and  $M_0^\prime = R(a_0)$.    } \\ 
\hline 
  [0.2] & \parbox{3.75in}{Identify $d_0 $ and $c_0$, if they exist.  } \\ 
  \hline
  [0.3] & \parbox{3.75in}{Identify   $a_1, a_1^\prime$ and if they exist, $c_1, d_1^\prime$.$^1$  } \\ 
\hline
 [0.4] & Identify one or both of  $b_1, b_1^\prime$, if they exist.$^2$   \\ \hline
 [0.5] & \parbox{3.75in} {Redefine $I(b_0) := (M_0,M_0^\prime)$.  If $d_0$ exists and $I(d_0)$ has white endpoints, redefine $I(d_0) := [M_0, M_0^\prime)$  and color its endpoints red. If $c_0$ exists and $I(c_0)$ has white endpoints, redefine $I(c_0) := (M_0, M_0^\prime]$\\ and color its endpoints red.} \\ \hline
 [0.6] & \parbox{3.75in}{Color the endpoints of $I(a_0), I(b_0)$ red.
 $^3$}  \\ \hline
 [0.7] & GO TO STEP [1.1].  \\ \hline
  &\parbox{3.75in}{ $^1$$L(a_1) < L(c_1)$ by Lemma~\ref{abc-properties}(\ref{peekers-not-contained})  and Theorem~\ref{alg-correct} parts (1a, 2). Similarly, $R(d_1') < R(a_1')$. }\\
  & $^2$$b_1 \neq c_1$ since $R(b_1) \le R(a_1) < R(c_1)$. Similarly, $b_1 \neq d_1^\prime$.\\
 & $^3 c_1, d_1^\prime$ may not exist. \\ \hline
\end{tabular}
\caption{Base step for one sweep of the SWEEP algorithm.} 
\label{base-step}
\end{center} 
 \end{table}

{\bf Left part of $S_i$, step $k$:} Table~\ref{left-k-step} 
specifies the operations for the $k$-th step of $S_i$. 

After the base step, we sweep to the left.  At the end of step $k-1$, the zone line $M_{k-1}$ has been established,   we have identified the one or two intervals, $I(a_k)$ and perhaps $I(c_k)$, that  
  peek into 
$a_{k-1}b_{k-1}$ from the left (into $a_{k-1}c_{k-1}$ if $b_{k-1}$ doesn't exist). 

First we   redefine $R(a_k)$ to equal $M_{k-1}$ and color it red.
There are three possibilities for what happens next, and two of them result in the end of the left part of the sweep. 

\begin{itemize}

\item If neither $b_k$ nor $c_k$ exists,  we say $a_k$ is a  \emph{terminal {\bf a}}. 

\item If $c_k$ exists and $L(a_k)$ is red from an earlier sweep, we say $a_k$ is a \emph{merging \bolda} and that $c_k$ is a \emph{merging {\bf c}}. In this case, modify $I(c_k)$   to match $I(a_k)$ and color its endpoints red.

If $a_k$ is a terminal or merging {\bf a}, we say the left part of $S_i$ \emph{ends} and we begin the right part of $S_i$. This occurs in step $[k.1]$ of Table~\ref{left-k-step}.

\item Otherwise, we continue step $k$. We color $L(a_k)$ red, establish a zone line $M_k = L(a_k)$, and identify $d_k$ if it exists.
At least one of $b_k, c_k$ exists. A vertex peeks into $a_kb_k$ from the left by Hypothesis~\ref{hyp}, or into $a_kc_k$ from the left by Lemma~\ref{abc-properties}(6) and because $G$ is twin-free. In preparation for step $k+1$ we can now identify $a_{k+1}$ and, if they exist, $c_{k+1}$ and $b_{k+1}$.

 We modify the intervals that exist for $b_k,c_k $  to match $I(a_k)$   and color their endpoints red.   The interval for $d_k$ is moved only if its endpoints are white and if $I(d_{k-1})$ either matches $I(a_{k-1})$ or does not exist.  
 We then continue sweeping left and go to step $[k+1.1]$ of Table~\ref{left-k-step}.
 
 It will help to note that in these steps, any endpoints that are redefined are changed as follows: $R(a_k), L(b_k), L(c_k), R(c_k)$ move to the left while $R(b_k), L(d_k), R(d_k)$ move to the right.

\end{itemize}

\begin{sidewaystable}
\begin{center}
  \begin{tabular}{| l | l  | | l |l |}
  \hline
  Step & Operation & Step & Operation \\ \hline
  [$k$.1] & \parbox{2.9in}{Redefine $I(a_k) := [L(a_k), M_{k-1}]$ and 
  color $R(a_k)$ 
  red. If $a_k$ is a merging {\bf a}, redefine $ I(c_k) := (L(a_k),M_{k-1}]$ and color its endpoints red. 
  If $a_k$ is a terminal  or merging {\bf a}, GO TO STEP [$1.1^\prime$], and begin a right sweep.} 
  &  
  [$k.1^\prime$] & 
  \parbox{2.9in} {Redefine $I(a_k^\prime) := [M_{k-1}^\prime, R(a_k^\prime)]$ and
 color $L(a_k^\prime)$ 
  red. 
  If $a_k^\prime$ is a merging {\bf a}, redefine $I(d_k^\prime) := [M_{k-1}^\prime, R(a_k^\prime))$ and color its endpoints red.
   If $a_k^\prime$ is a terminal  or merging {\bf a}, GO TO STEP $[0.1]$ and begin a new complete sweep.}
  
      \\ \hline 
      [$k$.2] & \parbox{2.9in} {Otherwise, color $L(a_k)$ red, establish zone line $M_k = L(a_k)$, and
 identify  $ d_k$ if it exists.$^1$ \strut}& [$k.2^\prime$] &  
\\ \cline{1-3}
[$k$.3] & \parbox{2.9in}{Identify $a_{k+1}$ and, if they exist,  $c_{k+1}$ and $b_{k+1}$.$^2$}  
& 
[$k.3^\prime$] & 
\\ \cline{1-3}

[$k$.4] &\parbox{2.9in}{ If they exist, redefine $  I(b_k) := (M_k,M_{k-1})$, $I(c_k) := (L(a_k),M_{k-1}]$ and color their endpoints red.$^3$ \strut} & 
[$k.4^\prime]$ & 
\parbox{2.9in}{In [$k.2$]-[$k.5$], replace `left' by `right' and vice versa; replace $a_k, b_k$ by $a_k^\prime, b_k^\prime$ ;  $c_k$ by $d_k^\prime$, $d_k$ by $c_k^\prime$;  $M_k$ by $M_k^\prime$.}
\\ \cline{1-3}

[$k$.5] & \parbox{2.9in}{ If $d_k$ exists and   $I(d_k)$ has white endpoints, and either  $d_{k-1}$ does not exist or $I(d_{k-1}) = [M_{k-1},M_{k-2})$,  redefine $ I(d_k) := [M_k,M_{k-1})$  \& color its endpoints red.$^4$}& 
 [$k.5^\prime$] & \\ \hline

[$k$.6] &  \parbox{2.9in}{
GO TO STEP [$k+1$.1].\strut} &
 [$k.6^\prime$]& 
 \parbox{2.9in}{
 GO TO STEP [$k+1.1^\prime$].} \\ \hline

 & \parbox{2.9in}{$^1$There is at most one choice for $d_k$ since $G$ is twin-free.}& &\\
&\parbox{2.9in}{$^2$ $L(a_{k+1}) < L(c_{k+1})$ by Lemma~\ref{abc-properties}(\ref{peekers-not-contained})  and Theorem~\ref{alg-correct} parts (1a, 2). Also $b_{k+1} \neq c_{k+1}$ since  $R(b_{k+1}) \le R(a_{k+1}) < R(a_{c+1}).$} & &\\
&$^3$At least one exists since $a_k$ is not a terminal {\bf a}. & &\\
&$^4$If $k=1$, interpret $M_{k-2}$ as $M_0^\prime$.&
& \\
\hline
 \end{tabular}
\caption{$k$-th step of the left and right parts of a complete sweep, $k\ge 1$. 
} 
\label{left-k-step}
\end{center} 
 \end{sidewaystable}

 \medskip
 \noindent {\bf Right part of $S_i$}:
  Analogously, in the right part of $S_i$ the roles of vertices  of types {\bf c, d} are reversed. 
  When the right part terminates in a step of the form $[k.1^\prime]$, the sweep begun with $a_0b_0$ is complete. In the right part of a sweep, $L(a^\prime_k), R(b^\prime_k), L(d^\prime_k),R(d^\prime_k)$ move to the right and $L(b^\prime_k), L(c^\prime_k), R(c^\prime_k)$ move to the left.
 
 \medskip  
 \noindent {\bf Termination}:   Each complete sweep of the algorithm reduces the number of strict inclusions  among the intervals. After  a complete sweep $S_i$, if any strict inclusions remain we identify a new $ab$-pair and begin a new sweep $S_{i+1}$. When no strict inclusions remain the algorithm terminates, and it will follow from Theorem~\ref{alg-correct}(4) that we have produced a strict mixed interval representation of $G$.

}
\end{alg}

  \subsection{Proof of Correctness of the SWEEP algorithm}
  \label{Completing}

   We now prove that  Algorithm~\ref{sweep} correctly produces a strict interval representation of the input graph $G$. We carry out the proof by induction on the number of complete sweeps of the algorithm. We will   
prove that once the endpoints of an interval are colored red, the interval remains unchanged until the algorithm terminates. Also, if the endpoints of one or more intervals are white after a sweep then these intervals have not changed from the  initial closed representation of $G$ given in Hypothesis~\ref{hyp}. Therefore they satisfy Proposition~\ref{claims} and Lemma~\ref{abc-properties}, and cannot induce in $G$ a graph in the forbidden set ${\cal F}$.   
  
\begin{defn}
\label{participates} {\rm  An endpoint of an interval $I(v)$ \emph{participates in sweep $S$} if $v$ is identified as being of type {\bf a}, {\bf b}, {\bf c}, or {\bf d} in $S$ and after any modification during $S$, the endpoint lies on a zone line established in $S$.    In addition, left endpoints of merging {\bf c}'s and right endpoints of merging {\bf d}'s are also said to participate in  $S$.  
A vertex $v$ \emph{participates in sweep} $S$  if at least one   endpoint of $I(v)$ participates in $S$.

}
\end{defn}

During the left part of the sweep, for vertices identified as a terminal {\bf a} or a merging  {\bf a}, only   right endpoints participate, and for vertices identified as type {\bf d}, either both endpoints or neither endpoint participates  (see [k.5] in Table~\ref{left-k-step}).    For other vertices identified as being of type {\bf a} and all vertices   identified as being of type {\bf b} or   {\bf c}, both endpoints participate.

 The next remark, which we use repeatedly in the proof of Theorem~\ref {alg-correct}, follows from   the construction given in Algorithm~\ref{sweep}.  
   
  \begin{remark} 
  \label{alg-rem}

{\rm  (a) At any point in the algorithm, if an endpoint of $I(v)$ is white then the endpoint is in its initial position and is closed. 

(b) After any number of complete sweeps, if $I(v)$ has one red and one white endpoint then $v$ was a terminal {\bf a} in the sweep  that turned the endpoint red. }   \end{remark}
  
  Parts (2) and (4) of the next theorem are crucial to proving our algorithm terminates and that the graph represented by the intervals doesn't change.     The other parts are needed in the induction.

\begin{theorem}
\label{alg-correct}
       Let the input to Algorithm~\ref{sweep} be a graph $G$ with a representation  satisfying Hypothesis~\ref{hyp}, and let $S_i$ be the $i$th sweep of the algorithm.   For each $i \ge 1$,
\begin{enumerate}
\item[(1)] \label{claim-ab}
Suppose $I(u)$ is strictly contained in $I(v)$ at the beginning of $S_{i}$.
\begin{enumerate}
\item[(a)] \label{strict-incl}
All four endpoints of these intervals are white.
\item[(b)] \label{ab-incl-lem}
If  $x$ peeks into $vu$ from the left, then  $R(x)$ is white.   If $x$ peeks into $vu$ from the right then $L(x)$ is white.  
\end{enumerate}
\item[(2)] 
 \label{well-def}  The quantities defined in Definition~\ref{pairs-def} and the operations specified for sweep $S_i$ in Tables~\ref{base-step} and \ref{left-k-step}  
are well-defined. 

Each endpoint participating in sweep  $S_i$  is white at the beginning of $S_i$ and red at the end of $S_i$.
In particular, if an endpoint of $v$ was colored red before $S_i$ then it is not moved during $S_i$.

\item[(3)]   At the end of  $S_i$, the following are true for each zone line $M$  established during $S_j$,   $1 \le j \le i$:
  
  \begin{enumerate}[(a)]
  \item  Any endpoint on $M$ is red.
  \item There exist vertices  $a, \hat{a}$ of type {\bf a} participating in $S_j$ with $R(a) = L(\hat{a})  = M$.
    
 \item    Each interval with an endpoint on $M$ shares the same endpoints as an interval of type {\bf a} whose vertex participates in  $S_j$.  
  
  \item If $M$ is established in the left part of $S_j$ there exists an interval of type {\bf b}  or {\bf c}  whose left   endpoint is on $M$ and is open.  An analogous statement holds for the right part of $S_j$.
  
  \item  Each interval with an endpoint on $M$ either participates in $S_j$ or, for some sweep $S_\ell, j < \ell \le i$, is a merging {\bf c}({\bf d}) whose left (right) endpoint is moved to $M$ during $S_\ell$.

        \end{enumerate}
 
\item[(4)]  At the end of $S_i$, the graph represented by the intervals is $G$.

  \end{enumerate}

\end{theorem}
  
\noindent{\bf Proof of Theorem~\ref{alg-correct}.} The proof is by induction on $i$.   We begin with the base case $i = 1$.  In Lemma~\ref{abc-properties}(1), we   showed that the quantities $a_j,b_j,c_j,d_j$ are well-defined before any sweeps have been performed, establishing  the first sentence of (2)  in the base case. 
The remaining part  of (2) and statements    (1), (3) and (4)  are either trivially true or follow easily from the construction in the algorithm.

Now suppose statements (1)-(4) are  true for       $ i:  1 \le i \le r$. We must prove they are  true for   $i = r+1$, and we call $ S_{r+1}$,   the \emph{current sweep}.

 \smallskip

\medskip
 
 \noindent
{\bf Proof of (1a).}
  \  Suppose, to the contrary, that at least one endpoint of $I(u)$ or of $I(v)$ is red at the beginning of $S_{r+1}$. Let $S_i$ be the first sweep in which an endpoint turned red, where $1 \le i \le r$.  

\smallskip

\noindent \emph{Case 1.}  Suppose at least one endpoint of $I(u)$ and at least one of $I(v)$ turn red during $S_i$.   

First, suppose both endpoints of $I(v)$ and both endpoints of $I(u)$  turn red during $S_i$. Then $I(u)$ is not strictly contained in $I(v)$ at the end of $S_i$, since the algorithm never creates strict inclusions during a given sweep. Furthermore, by the induction assumption red endpoints are not moved by the algorithm between the end of  $S_i$ and the beginning of $S_{r+1}$, so $I(u)$ is not strictly contained in $I(v)$ at the beginning of $S_{r+1}$, a contradiction.

Next, suppose both endpoints of $I(v)$  and exactly one endpoint of $I(u)$ turn  red during $S_i$.   Without loss of generality, suppose $R(u)$ turns red and $L(u)$ remains white during $S_i$.  By the induction hypothesis (2) and Remark~\ref{alg-rem}, $u$ is  a terminal {\bf a} for the left part of $S_i$,   and after $S_i$ the right endpoint $R(u)$ lies on  the leftmost zone line $M$  of $S_i$. Since $v$ participates in $S_i$ and is not a terminal ${\bf a}$,  at the end of $S_i$ both endpoints of $I(v)$ lie on zone lines of $S_i$.   By the induction hypothesis (2), the positions of $L(v)$, $R(v)$  and $R(u)$ do not change between the end of $S_i$ and the beginning of $S_{r+1}$, so at the end of $S_i$, $L(v) = R(u) = M$.    In order to have $I(u) \subseteq I(v)$ at the beginning of $S_{r+1}$, the left endpoint   $L(u)$ is moved to $M$ during the right part of some sweep $S_{\ell}$ with $i < \ell \le r$.  Thus $u$ is a terminal or merging ${\bf a}$ for the right part of $S_{\ell}$.  So for $S_{\ell}$,  we have $u = a_k'$ for some $k$
   and there exists $a_{k-1}'$ with $R(a_{k-1}') = M$.   By (2),  $R(a_{k-1}')$ is white at the beginning of $S_{\ell}$ and thus also at the end of $S_i$ but 
lies on
    a  zone line of $S_i$,
contradicting (3a).

Thus,  exactly one endpoint of $I(v)$ turns red during $S_i$.
Without loss of generality, suppose $R(v)$ turns red and $L(v)$ remains white during $S_i$. By the induction hypothesis (2) and Remark~\ref{alg-rem}, $v$ is  a terminal {\bf a} for the left part of $S_i$ and after $S_i$, its right endpoint $R(v)$ lies on the leftmost zone line $M$ of $S_i$. 
Since $u$ participates in  $S_i$, is adjacent to $v$,  and is not a terminal ${\bf a}$, after $S_i$ both endpoints of $I(u)$ lie on zone lines of $S_i$.   By the induction hypothesis (2), the endpoints of $I(u)$ do not change position between the end of $S_i$ and the beginning of $S_{r+1}$ when $I(u) \subseteq I(v)$, so at the end of $S_i$, $L(u) = R(u) = M$.  Also by the induction hypothesis (3c), $I(u)$ has the same endpoints after $S_i$ as some $I(a_j)$, where $a_j$ participates in $S_i$.     Thus $|I(a_j)| = 0$ after $S_i$.  However, $L(a_j) < L(a_{j-1})$ 
initially, before any sweeps, by Hypothesis~\ref{hyp}. So during $S_i$, $L(a_j)$ is unchanged  and $R(a_j)$ is retracted (to the left) to $L(a_{j-1})$ .  Thus $|I(a_j)| > 0$ after $S_i$, a contradiction.

\smallskip

\noindent
\emph{Case 2.}  Suppose both endpoints of $I(u)$ are white at the end of $S_i$. 

Without loss of generality assume $R(v)$ turned red in $S_i$, lying on a zone line $M$ established during $S_i$.

First consider the case where $I(u) \subseteq I(v)$ at the end of $S_i$.  
By the construction in the algorithm,  there is a vertex $a$ of type {\bf a} (possibly equal to $v$) for which $R(a)$ is red and $I(a), I(v)$ have the same endpoints at the end of $S_i$.  But then we would have identified $au$ as an $ab$-pair and the endpoints of $I(u)$ would also have turned red during $S_i$, a contradiction.   Thus at the end of $S_i$, $I(u) \not\subseteq I(v)$ and either $R(u) > R(v) = M$ or $L(u) < L(v)$. 
 
Suppose $R(u) >   M$  at the end of $S_i$. 
 Since $I(u) \subseteq I(v)$ at the start of $S_{r+1}$, by the induction hypothesis (4) we know $u$ is adjacent to $v$ in $G$, thus at the end of $S_i$ (and indeed initially) we have $L(u) \le M$.  Since $R(v) = M$ and is red at the end of $S_i$, there  exists some $\ell$ with $i <\ell < r+1$ for which $u$ participates in $S_\ell$, with $R(u)$ retracted to the left and $R(u) \le M$ at the end of $S_\ell$.

 If $R(u)$ is retracted in the left part of $S_\ell$, then $u$ is of type {\bf a}  or {\bf c} and peeks into $a_kb_k$ (or $a_kc_k$) of $S_\ell$ for some $k$, and $R(u)$ is retracted to $L(a_k) \le M$ during $S_\ell$.  Indeed, $L(a_k) \neq  M$ by (3a) because $L(a_k)$ is white at the beginning of $S_\ell$, by induction hypothesis (2).   
 But then $v$ is not adjacent to whichever of $b_k$, $c_k$ exists before $S_\ell$ but  is adjacent to them after $S_\ell$, contradicting (4).

 If $R(u)$ is retracted in the right part of $S_\ell$, then $u$ is of type {\bf c} and is matched to some vertex $a_k'$.   In this case, $a_{k+1}'$ exists.    Both endpoints of $I(u)$ turn red in $S_\ell$ and do not move between the end of $S_\ell$ and the beginning of $S_{r+1}$.  Then at the end of $S_\ell$,  $R(a_{k-1}') = L(a_k') = L(u)$ and $ R(a_k') = R(u)$, and the right endpoints $R(a_k')$ and $R(a_{k-1}')$ do not move during $S_\ell$, by the construction in the algorithm.  We know $I(v)$ contains $R(u)$ at the end of $S_\ell$ and will show it also contains $L(u)$ then.  If $L(v)$ is red at the beginning of $S_\ell$ then $L(v) \le R(a_{k-1}')$ because $I(v)$ contains $I(u)$ at the beginning of $S_{r+1}$ and red endpoints don't move by (2).  
At the beginning of $S_\ell$, $R(a_{k-1}')$ is white by (2)  so it cannot lie on the zone line at $L(v)$, by (3a). Thus $L(v) < R(a_{k-1}')$, $I(v)$ intersects the intervals for $a_{k-1}', a_k'$ and $a_{k+1}'$, both at this stage and initially, by (4). But this contradicts Lemma~\ref{abc-properties}(\ref{old-AAA}). 
 
 If $R(u)$ is retracted in the base step of $S_\ell$, then $u = c_0$ and a minor modification of the preceding paragraph results in the same contradiction.

 Therefore $L(v)$ is white at the beginning of $S_\ell$ and by Remark~\ref{alg-rem}, $v$ was a terminal {\bf a} for the left part of $S_i$.  If $L(v)$ is moved in a later sweep, it can only be moved to the right as a terminal or merging {\bf a} for the right part of a sweep, so $I(v)$  intersects the intervals for $a_{k-1}', a_k'$ and $a_{k+1}'$ and we get the same contradiction.

Therefore at the end of $S_i$ we have $R(u) \le M$ and $L(u) < L(v)$.  Indeed, $R(u) \neq M$ at the end of $S_i$ by (3a) because $R(u)$ is white.  
At the end of $S_i$, if $L(v)$ is red we have a case symmetric to the one just considered so we may assume $L(v)$ is white then.  Since $R(v)$ is red, this implies that $v$ is a terminal $\bolda$ for the left part of $S_i$ and there exists $\hat{a}$ with $R(v) = L(\hat{a})$  at the end of $S_i$, by (3b).   Thus  if $v$  participates in another sweep prior to $S_{r+1}$, it will be   a terminal   or merging $\bolda$  for the right part of that sweep and   $L(v)$  will move  to the right, turn red, and not be moved again before the start of $S_{r+1}$.  Thus $L(u)$ must move to the right of  the initial position of $L(v)$ during some sweep $S_j$, where $i < j < r+1$.  There are two possibilities.

If $u$ participates in the right part of  $S_j$ then $u = a_k'$ or $u = d_k'$ for some $k \ge 1$.  In this case, there exists $a_{k-1}'$ participating in $S_j$ with right endpoint on a zone line $M'$ of $S_j$. Since $L(u)$ will move to $M'$, turn red in $S_j$, and not move again, we must have $L(v) \le M'$.  So there exists $b_{k-1}'$ or $d_{k-1}'$ so that $u$ peeks into $a_{k-1}'b_{k-1}'$ ($a_{k-1}'d_{k-1}'$).  But $v$ also peeks into this pair, so at the end of $S_j$ the intervals for $u$ and $v $ will have the same endpoints, which will be red and not move between $S_j$ and  $S_{r+1}$. This contradicts the assumption that $I(u)$ is strictly contained in $I(v)$ at the start of $S_{r+1}$.  

Thus    $u$ participates in the left part  or base step of $S_j$. Since $L(u)$ moves to the right, we must have $u = d_k$ for some $k$ and there exists $a_k$ participating in $S_j$. We will show  $I(a_k) \subset I(v)$ at the start of $S_i$ and that this is a contradiction.

After $S_j$, we have $R(a_k) = R(u) <  M$  because $u$ is not adjacent to $\ahat$ after $S_i$ and therefore after any sweep, by the induction hypothesis (4). Before $S_j$, the endpoints of $I(a_k)$ are white  and are in their initial positions. Thus before $S_i$ we have $R(a_k) <  M$, and $M \le R(v)$ since $R(v)$ can only move to the left during $S_i$.
Also, $L(v)$ can only move to the right after $S_i$. Before $S_j$ the endpoints $L(a_k), L(u)$ are in their initial positions, and after $S_j$ they are equal and do not move afterward. Since $L(u) < L(v)$ at the end of $S_i$ and $L(v) < L(u)$ at the beginning of $S_{r+1}$, at the start of $S_i$ we must have $L(v) < L(a_k)$.

Thus $I(a_k) \subset I(v)$ at the start of $S_i$.   But then $a_k$ would have participated in $S_i$ as a type $\boldb$ interval  contained in $I(v)$, and $v$ would not have been a terminal $\bolda$, a contradiction.  This completes Case 2.

 \smallskip
  
  \noindent
 \emph{Case 3.} Both endpoints of $I(v)$ are white at the end of $S_i$. 
 
 Without loss of generality, we may assume $R(u)$ turns red during $S_i$.  
 
 First, suppose both endpoints of $I(u)$ turned red during $S_i$. Therefore, at the start of $S_{r+1}$,  the interval $I(v)$ crosses two zone lines that were established in $S_i$, and thus vertex  $v$ is adjacent to three vertices of type {\bf a} from $S_i$. 
 By the induction hypothesis (4), the graph represented by the intervals at the beginning of sweep  $S_{r+1}$ is the original graph $G$, so $v$ was initially adjacent to these vertices. Since we can choose to begin the algorithm with $S_i$ as the first sweep, these three vertices were of type ${\bf a}$ in the initial closed representation, which contradicts Lemma~\ref{abc-properties}(\ref{old-AAA}).

 Therefore, $R(u)$ turns red during $S_i$ while $L(u)$ remains white.   By the induction hypothesis (2) and Remark~\ref{alg-rem}, 
 $u$ is a terminal $\bolda$ for the left part of   $S_i$ and after $S_i$, the right endpoint   $R(u)$ lies on the leftmost zone line $M$ of $S_i$.
  Thus during   $S_i$, we have $u = a_{k+1}$ for some $k \ge 0$, at least one of $b_k, c_k$ exists, and $u$ peeks into $a_kb_k$ ($a_kc_k$) from the left.  
 At the end of $S_i$, the endpoints $L(a_k), R(u)$ both equal $M$, are red, and do not move between $S_i$ and  $S_{r+1}$.  Since $I(v)$ contains them at the beginning of $S_{r+1}$, at the end of $S_i$ vertex $v$ must be adjacent to both $u$ and  $a_k$  and so $I(v)$ must contain $L(a_k) = R(u)$.  If $R(v) < L(b_k)$ ($R(v) < L(c_k)$)   at the beginning of $S_i$, then $v$ would peek into $a_kb_k$ ($a_kc_k$) during $S_i$ and would have participated in $S_i$, a contradiction.  Hence $v$ is adjacent to  any of $b_k,c_k$ that exist.
 We also know that $L(v) \le M$ at the end of $S_i$ (and originally) since $v$ is adjacent to $u$ in $G$, and furthermore, if $L(v) = M$ then 
 $I(v), I(a_k)$   have the same left endpoint initially, contradicting Proposition~\ref{claims}. Thus $L(v) < M$ initially. 
 
  If $L(v) \le L(u)$ initially  then 
  $I(u) \subseteq I(v)$ before $S_i$ and therefore initially.  This implies the existence of a vertex $x$ peeking into $vu$ from the left   initially.  If $b_k$ exists, then $v, a_k, x, u, b_k,a _{k-1}$ would induce the forbidden graph ${K_{2,4}}^*$ in $G$. If $c_k$ exists, then $a_{k-1},a_k,c_k, u, x, v$ would induce in $G$ a forbidden graph from Family 4. (In both cases, if $k=0$ replace $a_{k-1}$ by $a_1^\prime$.)  Each of these cases leads to a contradiction.

Thus we are left with the instance  $L(u) < L(v)$ initially. 
Since $I(u) \subseteq I(v)$ at the start of $S_{r+1}$, one or both of $L(u), L(v)$ must move between the end of $S_i$ and the beginning of $S_{r+1}$ and by the induction hypothesis, each endpoint can move (and turn red) at most once.   The endpoints $L(u)$ and $L(v)$ must participate in different sweeps for $I(u)$ to be strictly contained in $I(v)$ at the start of $S_{r+1}$.

 First consider the case in which $L(u)$ participates before $L(v)$, that is, $L(u)$  
participates in a sweep $S_p$ with $i < p < r+1$ and $L(v)$ remains white at the end of $S_p$.  Then $u$ is a terminal or merging 
 $\bolda$ for the right part of $S_p$, and thus $u=a_{j+1}'$ for some $j>0$.  Thus there exist  $a_j'$, a zone line $M'$, and at least one of $b_j'$, $d_j'$  in $S_p$ with $R(a_j') = L(u) = M'$ at the end of $S_p$.   If $M' \ge L(v)$ then $v$ would peek into $a_j'b_j'$ (or $a_j'd_j'$) and would have participated in $S_p$.   Hence  $M' < L(v)$ at the end of $S_p$ and  we have $R(a_j') < L(v)$ and hence $v$ is not adjacent to $a_j'$ in $G$.  Between the end of $S_p$ and the beginning of $S_{r+1}$, the left endpoint of $v$ must move so that $L(v) \le L(u) = R(a_j')$, and this means $v$ will be adjacent to $a_j'$ at the beginning of $S_{r+1}$, contradicting the induction hypothesis (4).

Next, consider the case in which $L(v)$  participates before $L(u)$, that is, $L(v)$ 
participates in a sweep $S_q$ with $i < q < r+1$,  and $L(u)$ remains white at the end of $S_q$.   Suppose that after $S_q$ we have $L(v) \le L(u)$.  During $S_q$, there exists some $\ahat$ of type $\bolda$ participating in $S_q$ for which, after $S_q$, interval $I(v)$ has the same endpoints as $I(\ahat)$.  We know $\ahat$ is adjacent to $b_k$ ($c_k$) because $v$ is adjacent to $b_k$ ($c_k$) and $R(\ahat)$ is closed.  Thus $R(\ahat)> R(u)$ even initially.  Intervals of type $\bolda$ are never expanded, so initially, $I(u) \subseteq I(\ahat)$ and there exists a vertex $x$ that peeks into $\ahat u$ initially from the left.  If $b_k$ exists then there exists a vertex $y$ that peeks into $a_kb_k$ from the right and the vertices $a_k,\ahat,y,u,b_k,x$ induce in $G$ a $K_{2,4}^*$, a contradiction.  Otherwise, $c_k$ exists and from sweep $S_i$ we get a forbidden graph from Family 4 with tail induced by $a_{k-1}, a_k,c_k, u,x,\ahat$, a contradiction.

Therefore,  $L(v)$ turns red before $L(u)$ in sweep $S_q$ but $L(u) < L(v)$  after $S_q$ and in a subsequent sweep $S_p$, where $i < q < p < r+1$, endpoint   $L(u)$ moves to the right.
Then $u = a_j'$ for some $j$ and is a merging or terminal $\bolda$ for the right part of   $S_p$. Thus $a_{j-1}'$ exists,   one of $b_{j-1}^\prime, d_{j-1}^\prime$ exists in $S_p$, and $u$ peeks into $a_{j-1}'b_{j-1}'$ ($a_{j-1}'d_{j-1}'$) from the right.  Then $L(u)$ is retracted to $R(a_{j-1}')$ and $v$ must be adjacent to $a_{j-1}'$, because after $S_p$ we have $L(u) = R(a_{j-1}') \in I(v)$.  So $v$ peeks into $a_{j-1}'b_{j-1}'$ ($a_{j-1}'d_{j-1}'$) during $S_p$  and $L(v)$ would have participated in $S_p$.   But $L(v)$ is red at the beginning of $S_p$, contradicting the induction hypothesis (2).

 This completes Case 3 and the proof of  induction hypothesis (1a).  

\medskip
\noindent
{\bf Proof of (1b).}  We prove the first statement;  the second is analogous.  For a contradiction, suppose $R(x)$ is red at the start of sweep $S_{r+1}$.  
  Then it turned red during  an earlier sweep $S_i$, and  by the construction in the  algorithm, there exists a zone line $M$ at $R(x)$ that was established during $S_i$.  
  
  There is a vertex $y$ of type {\bf a} that participates in $S_i$ and has $L(y)$ red and equal to $M$.  Indeed,     by induction hypothesis (2), $L(y)$ does not move between the end of $S_i$ and the beginning of $S_{r+1}$. So 
  $L(v) < L(y) < L(u)$.
 If $R(y) < R(v)$ then $I(y)$ is strictly contained in $I(v)$, while if $R(y) \ge R(v)$ then $I(u)$ is strictly contained in $I(y)$. Both possibilities contradict induction hypothesis (1a).

  \medskip
 
    \noindent
 {\bf Proof of (2).}  For sweep $S_{r+1}$ we verify that the quantities defined in Definition~\ref{pairs-def}
 and the
  operations specified  in Tables~\ref{base-step} and \ref{left-k-step}  
are well-defined.  In addition, we verify that all participating endpoints are white at the beginning of $S_{r+1}$.

\smallskip

\noindent
{\bf Steps in Table~\ref{base-step}:}   We identify an $ab$-pair $a_0b_0$ in $[0.1]$.  In $[0.2]$ we   identify $c_0$ and $d_0$ if they exist;   at most one of each can exist   since $G$ is twin-free, and by induction hypothesis (4) the graph represented is $G$.  If $c_0$ or $d_0$ participates in $S_{r+1}$, then by the construction in the algorithm its participating endpoints are white at the beginning of $S_{r+1}$.  

    We next show there exists a vertex $x$ peeking into $a_0b_0$ from the left as needed in $[0.3]$.  
 By induction hypothesis (1a), all four endpoints of $I(a_0), I(b_0)$ are white and by 
 Remark~\ref{alg-rem}(a), 
  these intervals are in their initial positions.  Thus by Hypothesis~\ref{hyp}, there exists an $x$ that initially peeked into $a_0b_0$ from the left, so $x$ is adjacent to $a_0$ but not $b_0$ in $G$.    By  induction hypothesis (4),  these adjacencies haven't changed so $x$ must peek into $a_0b_0$ at the start of sweep $S_{r+1}$.  If $x$ peeks into $a_0b_0$ from the right then $I(x)$ has been moved by the algorithm and its endpoints are red at the beginning of $S_i$, contradicting induction hypothesis (1b).  Thus $x$ peeks into $a_0b_0$ from the left at the start of $S_{r+1}$ and, by induction hypothesis (1b), $R(x)$ is white.  There are at most two such vertices, for otherwise the forbidden graph $B$ would be induced in $G$, a contradiction.  If there are two such vertices, their right endpoints are in their initial positions by Remark~\ref{alg-rem}(a) and hence unequal by Hypothesis~\ref{hyp}.  Thus $a_1$ and $c_1$ (if it exists) are well-defined and $R(a_1) < R(c_1)$.

We show that if both $a_1$ and $c_1$ exist in [0.3] then $L(a_1) < L(c_1)$ as noted in Table~\ref{base-step}.  For a contradiction, assume $L(c_1) \le L(a_1)$ thus $I(a_1)$ is strictly contained in $I(c_1)$.  By induction hypothesis (1a) all four of these endpoints are white thus by Remark~\ref{alg-rem}(a),
$I(a_1)$ and $I(c_1)$ are in their initial positions.  Now Lemma~\ref{abc-properties}(4) applies to these intervals and we get a contradiction.  The arguments for $a_1^\prime$ and $d_1^\prime$ are similar.

If  there exists a vertex $u$ for which $I(u)$ is strictly contained in $I(a_1)$, then all four of these endpoints are white by induction hypothesis (1a) and are in their initial positions by Remark~\ref{alg-rem}(a), so Proposition~\ref{claims}(1) implies that there is at most one such $I(u)$.  Hence $b_1$, and similarly $b_1^\prime$, is well-defined in $[0.4]$.  (If $L(a_1)$ is red at the beginning of $S_{r+1}$ then $a_1$ is a terminal ${\bf a}$  or a merging ${\bf a}$ for $S_{r+1}$ and $L(a_1)$ does not participate in $S_{r+1}$.)

 It remains to prove that $L(c_1)$ is white  at the start of $S_{r+1}$.
Suppose $c_1$ exists   and $L(c_1)$ is red.     Since $R(c_1)$ is white, Remark~\ref{alg-rem}(b)   implies that $c_1$ was a terminal {\bf a} for the right part of an earlier sweep $S_j$ beginning with an inclusion $a_0'b_0'$. Then for some $n \ge 0$, $c_1$ is $a_{n+1}^\prime$ in $S_j$. So at the end of $S_j$ there exist intervals $I(a_n^\prime)$ and either $I(d_n^\prime)$ or $I(b_n^\prime)$ with right endpoints equal to $L(a_{n+1}')$, where $I(a_n^\prime)$ is closed on the right and the other is open on the right. If $n=0$, we interpret $a_{-1}'$ to be a vertex that peeks in to $a'_0b'_0$ from the left. If $b_n^\prime$ exists then $a_0,a_1,c_1,b_n^\prime,a_{n-1}^\prime,a_n^\prime$ induce in $G$ the tail of a graph in Family 4.   
 Thus $n>0$ and  $d_n^\prime$ exists,  and then the two sweeps meeting at  $a_1,c_1,a_n^\prime,d_n^\prime$ induce a graph in Family 5. Again we have a contradiction, and thus $L(c_1)$ is also white at the start of $S_{r+1}$.

\smallskip

\noindent
{\bf  Steps in Table~\ref{left-k-step}: }
 
Now assume that at the beginning of $S_{r+1}$, all quantities up through Step $[k.1]$ are well-defined and all participating endpoints were white.
No new quantities are defined in step [k.1].
If the sweep continues to step [k.2], we establish a zone line $M_k = L(a_k)$ and  identify $d_k$ if it exists.  There can be only one such vertex $d_k$ since $G$ is twin-free and the graph represented by the intervals is $G$ by induction hypothesis (4).   By the construction in the algorithm, if $d_k$ participates, its endpoints are white at the beginning of $S_{r+1}$.

Next we show that $a_{k+1}$ exists, as required in step $[k.3]$.
  If $b_k$ exists, then the argument we used to show $a_1$ existed in step $[0.3]$ also applies and as before, $R(a_{k+1})$ is white at the beginning of $S_{r+1}$.  If $b_k$ does not exist and we have reached step $[k.3]$ then, because $a_k$ and $c_k$ are not twins, Lemma~\ref{abc-properties}(6) implies that there exists a vertex $x$  that peeks into $a_kc_k$ from the left.    There are at most two such vertices, for otherwise a forbidden graph from Family 1 would be induced in $G$, a contradiction. One such vertex $x$ is $a_{k+1}$ and the other, if it exists, is $c_{k+1}$.
  
   We will next show $R(x)$ is white.   
   For a contradiction, suppose $R(x)$ is red at the start of sweep $S_{r+1}$.  Then it turned red during an earlier sweep $ S_j$ and by the construction in the algorithm, there exists    a zone line $M$ at $R(x)$ that was established during $ S_j$.  
 Induction hypothesis (3b), implies that there exists a vertex $\hat a$ with $L(\hat a)$ red and equal to  $M$.   Thus $L(a_k) \le L(\hat a) < L(c_k)$ and in fact, by the induction hypothesis (3a), $L(a_k) < L(\hat a)$ because $L(a_k)$ is white, since $a_k$ is not a merging {\bf a}.  By induction hypothesis (1a) we must have $R(a_k) < R(\hat a) < R(c_k)$.    By the induction hypothesis (4), the graph at the beginning of sweep $S_{r+1}$ is the original graph $G$.  If $k \ge 2$ and $c_{k-1}$ exists,   the vertices $a_{k-2}, a_{k-1}, c_{k-1}, a_k,c_k, \hat a$ induce in $G$ the tail of a forbidden graph in Family 1.  Thus either  $k=1$ or $c_{k-1}$ does not exist. In both cases, $b_{k-1}$ exists, there is a vertex  $z$ that peeks into $a_{k-1}b_{k-1}$ from the right, and $a_{k-1},  a_k,c_k, \hat a, b_{k-1}, z$ induce the forbidden graph $B$, a contradiction.  Thus $R(x)$ is white, i.e., $R(a_{k+1})$ and $R(c_{k+1})$ (if it exists) are white. 

  The remaining arguments used in  showing $a_1,b_1,c_1$ are well-defined and in justifying the steps of $[0.3]$ can be applied to $a_{k+1}, b_{k+1}$ and $c_{k+1}$ and to the steps of $[k.3]$  as well.  By Lemma~\ref{abc-properties}(\ref{ac-ab-peekers}), and the induction hypothesis (4), any vertex that peeks into $a_kc_k$ from the left also peeks into $a_kb_k$ from the left.    If $b_k$  exists, then the same arguments we used to show the endpoints of $a_1$ and $c_1$ are white also apply to $a_{k+1}$ and $c_{k+1}$.
 So assume $b_k$ does not exist.  If $a_k$ is a terminal ${\bf a}$ or a merging ${\bf a}$ for the left part of $S_{r+1}$, there is nothing more to prove.

 So we may assume $c_k$ and $a_{k+1}$ exist and possibly also $c_{k+1}$. We must prove that 
$L(a_{k+1})$  and $L(c_{k+1})$  are also white at the start of  $S_{r+1}$.  If $b_{k+1}$ exists, its endpoints are white at the start of $S_{r+1}$ by $(1a)$.     If $L(a_{k+1})$ is red at the beginning of $S_{r+1}$ then $a_{k+1}$ is a   merging  ${\bf a}$ or a terminal ${\bf a}$  for $S_{r+1}$ and $L(a_{k+1})$ does not participate in $S_{r+1}$, so there is nothing to prove.  If $c_{k+1}$ exists then $L(c_{k+1})$ is white by the same argument used to show $L(c_1)$ is white (here, $a_1,c_1, a_0$ are replaced by $a_{k+1},c_{k+1}, a_{k }$).    If $d_{k+1}$ exists and participates in $S_{r+1}$ then its endpoints are white at the beginning of $S_{r+1}$.
 This proves (2) for the left part of $S_{r+1}$, and the right part is analogous.

    \bigskip

\noindent{\bf Proof of (3).} 
Consider the current sweep $S_{r+1}$.

(a) If $M$ was established during $S_j$ where $1 \le j < r+1$,  then (3a) holds at the beginning of $S_{r+1}$ by the induction hypothesis, and by part (2) these red endpoints on $M$ are not moved during $S_{r+1}$. 
Otherwise, $M$ was established (without loss of generality) during the left part of sweep $S_{r+1}$.   Let $a_k$ be the vertex of type ${\bf a}$ participating in $S_{r+1}$ with $L(a_k) = M$.  Since left endpoints of vertices of type ${\bf a}$ do not move during left parts of sweeps, $L(a_k) = M$  at the beginning of  $S_{r+1}$  and hence by the induction hypothesis, $L(a_k)$ was white at the beginning of $S_{r+1}$ and in its initial position.  If there exists another vertex with a white endpoint at $M$  after $S_{r+1}$, then that endpoint is also in its initial position and initially shared an endpoint with $a_k$, contradicting Hypothesis~\ref{hyp}.

(b)  If $M$ was established during $S_j$ where $1 \le j < r+1$,  then (3b) holds at the beginning of $S_{r+1}$ by the induction hypothesis, and by part (2) these red endpoints on $M$ are not moved during $S_{r+1}$.    If $M$ was established during $S_{r+1}$ then such an $a, \ahat$ exist by the construction of the algorithm.

 (c)  and (d)  These follow  directly from (b) and the construction in the algorithm.
 
 (e)  For each $M$ established during $S_j$ where $1 \le j < r+1$, (3e) holds at the beginning of $S_{r+1}$ by the induction hypothesis.   By part (2), the red endpoints of these participating intervals are not moved during $S_{r+1}$.  If the left endpoint of an interval is moved to $M$ during $S_{r+1}$ then it is a merging ${\bf c}$, for otherwise an interval of type $\bolda$ would have a white left endpoint on $M$ at the beginning of $S_{r+1}$, contradicting (3a).     A similar result holds if the right endpoint of an interval is moved to $M$ during $S_{r+1}$. This completes the proof of (3).

 \bigskip
 
Before we can present the proof of (4) we need several technical lemmas.  The algorithm allows sweeps to be done in any order and this affects  the sweep in which  an interval will participate.  A vertex that peeks into an $ab$-pair from the right could be of type {\bf d} either as $d_k^\prime$ in the right part of some sweep or as $d_k$ in the left part of a different sweep, depending on which  is done first.   Likewise, when a sweep $S$ merges with an earlier sweep $S'$, some intervals that would have participated in $S$ had it been performed earlier have instead participated in $S'$.  Lemmas~\ref{xd-lemma} and  \ref{merging-lem}
 consider  these possibilities and ensure that intervals can be moved as required by the algorithm without changing the graph represented.

  \begin{lemma}
   \label{xd-lemma}
      Suppose $a_{k+1}b_{k+1}$ is an $ab$-pair in the left part of sweep $S_{r+1}$.  If $x$  peeks into $a_{k+1}b_{k+1}$ from the right and $x \ne a_k$, then $x = d_k$.  An analogous result  holds for $ab$-pairs in the right part of $S_{r+1}$. 
   \end{lemma}
   
   \begin{proof}
    For a contradiction, suppose $x \ne d_k$. Setting $x_k = x$, we will show there exist vertices $x_k, x_{k-1}, \ldots, x_0$ such that the following are satisfied before and after $S_{r+1}$, for each  $j = k, k-1, \dots, 0$:
    
    \smallskip
    
   (a)  $I(x_j)$ lies completely to the right of $I(x_{j+1})$.
   
   (b)  $x_j$  is not identified as a vertex of type ${\bf a}, {\bf b},{\bf c},{\bf d}$ in the left part of   $S_{r+1}$.
   
   (c)  $I(x_j)$  crosses the zone line $M_j = L(a_j)$ established during  $S_{r+1}$, so $x_j$ is adjacent to $a_j$ and $a_{j+1}$.
   
   (d)  Every vertex  that is adjacent to $x_j$ is adjacent to $a_j$.
    
        \smallskip

      \noindent
 At the beginning of $S_{r+1}$, we first verify these conditions for $j = k$, where (a) is true vacuously.  To establish (b), note that $x_k$ is not equal to any of $a_{k+1}, b_{k+1}, c_{k+1}, d_{k+1}, b_k, c_k$ because $I(x_k)$ intersects $I(a_{k+1})$ but not $I(b_{k+1})$, and $x_k \neq a_k, d_k$ by assumption. If $x_k$ were any other vertex participating in the left part of $S_{r+1}$, then it would violate Lemma~\ref{abc-properties}(\ref{old-AAA}). For the same reason, $R(x_k) < R(a_k)$. Then (c) follows because  $I(x_k) \not\subseteq I(a_k)$, so $L(x_k) < L(a_k) = M_k$ and thus $I(x_k)$ crosses $M_k$.

 To show (d) holds for $x_k$, suppose to the contrary that  some vertex $v$  is adjacent to $x_k$ but not to $a_k$.    Then $I(v)$ must intersect $I(x_k)$  on the left and not intersect $I(a_k)$,  so $R(v) < M_k$.   Since $v \neq b_{k+1}$ and $I(v) \not\subseteq I(a_{k+1})$, we know $L(v) < L(a_{k+1})$.  But then $I(b_{k+1})$ is contained in both $I(a_{k+1})$ and $I(v)$.  By induction hypothesis (1a),   all these endpoints are white  and by Remark~\ref{alg-rem}(a) are in their initial positions, so this contradicts  Proposition~\ref{claims}.  Thus $x_k$ satisfies (a)-(d).
  
Now suppose $x_k, x_{k-1}, \ldots, x_j$ exist and satisfy (a)-(d) for some $j$, $1 \le j \le k$.
Since (b) implies $x_j \neq d_j$, it follows from Definition~\ref{pairs-def} and (d) that there exists a vertex $x_{j-1}$, different from $a_{j-1}$ and $d_{j-1}$, that is adjacent to $a_j$ but not $x_j$.     We will show (a)-(d) hold for $x_{j-1}$.

Since $I(x_j)$  crosses $M_j$, we have $L(x_j) < L(a_j)$ and so $I(x_{j-1})$ intersects $I(a_j)$ on the right and  is completely to the right of $I(x_j)$. This proves (a).

Next we show (b).  We know $x_{j-1} \neq a_j, d_j$ because $a_j$ and $d_j$ are adjacent to $x_j$ and $x_{j-1}$ is not.  If $x_{j-1} = b_j$ then $x_j$ would peek in to $a_jb_j$ from the left and violate (b) in case $j$, and the same reasoning shows $x_{j-1} \neq c_j$.     We know $x_{j-1} \neq a_{j-1},  d_{j-1}$ by assumption and $x_{j-1} \ne b_{j-1}, c_{j-1}$ because $x_{j-1}$ is adjacent to $a_j$. If $x_{j-1}$ were any other participating vertex it would violate Lemma~\ref{abc-properties}(\ref{old-AAA}).

To establish (c), suppose $I(x_{j-1})$ does not cross $M_{j-1}$.  Then either $x_{j-1} = b_{j-1}$, which we have just shown is false, or $R(x_{j-1}) > R(a_{j-1})$, violating Lemma~\ref{abc-properties}(\ref{old-AAA}).

Next, we establish (d).  Suppose there exists a vertex $v$ that is adjacent to $x_{j-1} $ but not to $a_{j-1}$.     First suppose $I(v)$ does not intersect $I(x_j)$.  Then we get a forbidden graph in Family 3 starting at $a_{k+1}b_{k+1}$, sweeping rightward, with tail induced by $a_{j+1}, a_j, x_j, a_{j-1}, v$, a contradiction.   Next suppose that $v$ is adjacent to $x_j$, but $I(v)$ is contained in $I(a_j)$.    In this case, we get a forbidden graph in Family 2 starting at $a_{k+1}b_{k+1}$, sweeping rightward, with tail induced by $a_{j+1}, a_j, x_j, a_{j-1}, x_{j-1}, v$, a contradiction.   Next suppose $L(v) < M_j$ but that $v$ is not adjacent to $x_{j+1}$.  Then 
we get a forbidden graph in Family 1 starting at $a_{k+1}b_{k+1}$, sweeping rightward, with tail induced by $a_{j+2},a_{j+1}, x_{j+1}, a_j, x_j,   v$, which gives a contradiction.    Finally, suppose $v$ is adjacent to $x_{j+1}$.  Then 
we get a forbidden graph in Family 4 starting at $a_{k+1}b_{k+1}$, sweeping rightward, with tail induced by $a_{j+2},a_{j+1}, x_{j+1}, x_j,      x_{j-1},v$, a contradiction.   

Applying this construction when $j = 1$ gives a vertex $x_0$ satisfying (a)-(d). In particular, every vertex adjacent to $x_0$ is adjacent to $a_0$.  We know $x_0$ is adjacent to $b_0$, for otherwise $x_0$ would peek into $a_0b_0$ and be identified as a participating vertex,  contradicting (b).  Since $x_0 \neq d_0$, there exists a vertex $y$ that is adjacent to $a_0$ but not to $x_0$ and by Definition~\ref{pairs-def}, $y$ must be adjacent to $b_0$.  Note $I(y) \not\subseteq I(a_0)$ by induction hypothesis (1a), Remark~\ref{alg-rem}(a), and    Proposition~\ref{claims} (parts (1) and (2)).   Now we get a forbidden graph from Family 2, sweeping rightward, with tail induced by $a_1,a_0,x_0,a_1', y, b_0$. 

Thus $x_0$ leads to a contradiction and we see that $x = x_k = d_k$. \qed

   \end{proof}

   \begin{lemma}
   \label{merging-lem}
   Let  $S_i$ be a sweep starting at $\hat a_0 \hat b_0$ whose right part has  a terminal {\bf a} at $\hat a_n'$. For $j  >  i$ let
    $S_j$ be a sweep starting at $a_0b_0$ whose left part has a merging {\bf a} at  $a_k$, and suppose
    $a_k = \hat a_n'$.  If $S_j$ were performed before $S_i$, it would terminate at $a_{k+m}$ for some $m\ge 0$ and for all of the following that exist, $a_{k+\ell} = \hat a_{n-\ell}'$, $b_{k+\ell} = \hat b_{n-\ell}'$, $c_{k+\ell} = \hat c_{n-\ell}'$.  The analogous result holds if $S_i$ is a left sweep and $S_j$ is a right sweep.
   \end{lemma}
    
     \begin{proof}
     One can check that $a_{k+\ell}$ of $S_j$ would be $\hat a_{n-\ell}'$ of $S_i$ and it follows directly that $b_{k+\ell} = \hat b_{n-\ell}'$.  When $\ell = n$ we write $\hat a_0$ instead of $\hat a_0^\prime$, etc. Now consider $c_{k+\ell}$ from $S_j$.   By Lemma~\ref{abc-properties}(6), no interval can intersect $I(c_{k+\ell})$ without also intersecting $I(a_{k+\ell})$.   Let $p$ be maximum so that $c_{k+p} \neq \hat c_{n-p}'$.  If $p=n$ then $c_{k+n} \neq \hat c_0$ and by Definition~\ref{pairs-def}, there exists some $v$ adjacent to $\hat a_0 = a_{k+n}$ but not to $c_{k+n}$. 
       We know  $\hat b_0 = b_{k+n}$ is adjacent to $c_{k+n}$ by  Lemma~\ref{abc-properties}(\ref{ac-ab-peekers}), so $v \neq \hat b_0$.  Furthermore,  $v$ does not peek into $\hat a_0 \hat b_0$ from the left by the definition of $\hat c_0$.  Thus $I(v)$ must intersect $I(\hat b_0) = I(b_{k+n})$.  This contradicts Lemma~\ref {abc-properties}(3). Thus $p<n$, and by definition of $\hat c_{n-p}'$, 
there exists a vertex $y$ other than $\hat a_{n-p-1}' = a_{k+p+1}$ and $\hat c_{n-p-1}' = c_{k+p+1}$, that meets $a_{k+p}$ but not $c_{k+p}$.     
  But then the three vertices $a_{k+p+1}, c_{k+p+1}, y$ peek into $a_{k+p}c_{k+p}$ from the left,   inducing a forbidden graph from Family 1 in $G$, a contradiction. \qed

      \end{proof}     
         
            \begin{defn}
      {\rm A vertex $d_k$ identified as a type ${\bf d}$ vertex    in the left part of sweep $S$ is called an \emph{NTM-d} (\emph{needs to move}) if either (i)    $b_{k+1}$ in exists in $S$ and $d_k$ peeks into $a_{k+1}b_{k+1}$ from the right or (ii) $d_{k+1}$ exists and is an NTM-$d$.  Analogously, we define NTM-$c$ vertices for the right part of  $S$.
       }
       \label{ntm-def}
      \end{defn}

      \begin{remark}
      For any vertex $d_k$ which is an NTM-d for sweep $S$, there exists an $ab$-pair $a_nb_n$ for some $n \ge k+1$ so that for each $j$ with $k \le j \le n-1$, vertex $d_j$ is an NTM-d for $S$ and $d_j$ peeks into $a_{j+1}d_{j+1}$. An analogous statement holds for NTM-$c$ vertices.
      \label{ntm-rem}
      \end{remark}
      
        \begin{lemma}
     \label{D-lemma}
     If $d$ is an NTM-d vertex in sweep $S_{r+1}$ then the endpoints of $I(d)$ are white at the beginning of $S_{r+1}$.
     The analogous result is true for NTM-$c$ vertices.
     \end{lemma}
     
     \begin{proof}
     For a contradiction, let $k$ be the maximum index for which $d_k$ is an NTM-$d$ vertex for $S_{r+1}$ and $I(d_k)$ has a red endpoint at the start of $S_{r+1}$.
   Suppose $L(d_k)$ is red, so that it turned red during a sweep $S_i$ for some   $  i \le r$ and there exists a zone line $M$ of $S_i$ with $M = L(d_k)$.  (Note that $d_k$ may not have been a type {\bf d} vertex for $S_i$.)

     The induction hypothesis for Theorem~\ref{alg-correct}(3b) implies that there exist  type ${\bf a}$ vertices $a, \ahat$ that participate  in $S_i$ and have $R(a) = L(\ahat) = M$.  If $L(d_k)$ is closed at the start of $S_{r+1}$ then $a$ is adjacent    to $d_k$ but not to $a_k$, contradicting Definition~\ref{pairs-def}.  
     Thus $L(d_k)$ is open.    If $b_{k+1}$ exists then $d_k$ peeks into $a_{k+1}b_{k+1}$ from the right, contradicting induction hypothesis (1b).  Thus by Definition~\ref{ntm-def}, $d_{k+1}$ exists and by the maximality of $k$, we know $R(d_{k+1})$ is white.  By induction hypothesis (3a), zone line $M$ lies strictly to the right of $I(d_{k+1})$.  But this induces a forbidden graph from Family 1 starting at $a_nb_n$ (specified in Remark~\ref{ntm-rem}) and sweeping rightward with tail $a_{k+2}, a_{k+1},d_{k+1}, a_k,d_k, \ahat$, a contradiction.

     Therefore,  $L(d_k)$ is white and $R(d_k)$ is red at the start of $S_{r+1}$ and we may assume that $R(d_k)$ turned red during some earlier sweep $S_i$, and lies on a zone line  of $S_i$.   Since $d_k$ was a terminal ${\bf a}$ for $S_i$, we know $R(d_k)$ is closed at the start of $S_{r+1}$.   By the induction hypothesis (3), there exists $\ahat_\ell$ and either $\bhat_\ell$ or $\chat_\ell$ participating in $S_i$ with $L(\ahat_\ell) = L(\bhat_\ell) = L(\chat_\ell)$.  Since $d_k$ is an NTM-$d$, we know there exists $n \ge 0$ for which $a_{n}b_{n}$ is an $ab$-pair and $d_j$ exists and is an NTM-$d$ for $k \le j \le n-1$.  If $\chat_\ell$ exists, we get a forbidden graph from Family  5  starting at $a_{n}b_{n}$ and sweeping rightward and  meeting the left part of $S_i$  at $a_k,d_k, \ahat_\ell,\chat_\ell$.  If $\bhat_\ell$ exists,  we get a forbidden graph from Family 4 starting  at $a_{n}b_{n}$  and sweeping rightward, with tail induced by vertices $a_{k+1}, a_k, d_k, \bhat_\ell, \ahat_{\ell-1}, \ahat_\ell$, a contradiction. 
     $\qed$
   
     \end{proof}
     
     \smallskip
     
    Now we can present the proof of Theorem~\ref{alg-correct}(4).

     \medskip
\noindent{\bf Proof of (4).}   Consider any two vertices $w,z$ of $G$.  
   Our goal is to show that the intervals assigned to $z$ and $w$  intersect prior to the current sweep $S_{r+1}$ if and only if they intersect after $S_{r+1}$.  This is certainly true if neither $w$ nor $z$ participates in $S_{r+1}$, and it is true by the construction in the algorithm if both $w$ and $z$ participate, so  without loss of generality, we may assume that $w$ participates in the left part of   $S_{r+1}$ and $z$ does not participate in this sweep.  
  
  We consider cases depending on the role $w$ plays in  $S_{r+1}$: it is either $a_k, c_k, b_k,$ or $d_k$ for some $k \ge 0$.    In each case we assume for a contradiction that the intervals for $z$ and $w$ intersect at the start of $S_{r+1}$ but not at the end, or vice versa.  We consider $k \ge 1$;  the arguments for $k = 0$ are analogous.  Let $M$ be the zone line $L(a_{k-1}) = R(a_k)$ at the end of $S_{r+1}$.
  
  First, suppose $w = a_k$, so $L(a_{k})$ is unchanged in $S_{r+1}$ and $R(a_{k})$ is retracted to the left to $M$.    We need only consider $I(z)$ intersecting $I(a_k)$ on the right at the start of $S_{r+1}$ and not intersecting afterwards.  Thus at the start of $S_{r+1}$ we have $M = L(a_{k-1}) < L(z) \le R(a_k)$.  If $R(z) \le R(a_{k-1})$ at the start of $S_{r+1}$ then $I(z) \subseteq I(a_{k-1})$ and $z$ would participate in $S_{r+1}$ as $b_{k-1}$, a contradiction.   Thus $R(z) > R(a_{k-1})$ and $z$ is adjacent to $a_k, a_{k-1}$ and $a_{k-2}$, contradicting Lemma~\ref{abc-properties}(\ref{old-AAA}) (where $a_{k-2}$ is interpreted as $a_1'$ if $k=1$).
  
Second, suppose $w = c_k$, so $L(c_k)$ and $R(c_k)$ are moved to the left.  If $I(z)$ intersects $I(c_k)$ before $S_{r+1}$ but not after, then $I(z)$ must intersect $I(c_k)$ on the right at the beginning of $S_{r+1}$.  By Lemma~\ref{abc-properties}(6), $I(z)$ also intersects $I(a_k)$   before $S_{r+1}$ but not after, contradicting the previous case.   Thus $I(z)$ does not intersect $I(c_k)$ before $S_{r+1}$ but does intersect it afterwards, so  before $S_{r+1}$ we have that $L(a_k) < R(z) < L(c_k)$ and $z$ peeks into $a_kc_k$ from the left.   If $a_k$ is not a merging {\bf a} for $S_{r+1}$  then $z$  participates in  $S_{r+1}$, a contradiction.  Therefore, $a_k$ is  a merging {\bf a} for $S_{r+1}$, $L(a_k)$ turned red during the right part of an earlier sweep $S_i, i \le r$, beginning at some $\hat a_0 \hat b_0$, and for some $n$, $a_k = \hat a_n^\prime$ was a terminal {\bf a} for $S_i$. If we were to perform $S_{r+1}$ after $S_{i-1}$ and before $S_i$ then, by Lemma~\ref{merging-lem} with $j = r+1$, $a_{k+1}$ and $z$ would be distinct vertices peeking into $a_kc_k$ and so $z$ would equal $c_{k+1} = \hat c_{n-1}'$. By reasoning as in Theorem~\ref{alg-correct}(2), we see all endpoints participating in $S_{r+1}$, including those of intervals of the form $I(c_{k+\ell}) = I(\hat c_{n-\ell}')$, would be white at the beginning of $S_{r+1}$ and thus were white at the end of $S_{i-1}$. So these endpoints were white and would have participated in $S_i$ for the original order of the sweeps, immediately after $S_{i-1}$. In particular, $z$ would be moved during $S_i$ and we would have $R(z) = L(a_k)$ at the end of $S_i$ and thus at the beginning of $S_{r+1}$, since red endpoints do not move. This contradicts $L(a_k) < R(z)$.

  Third, consider $w = b_k$.  Since $b_k$ is expanded on both sides during $S_{r+1}$, we need only consider $I(z) $ intersecting $I(b_k)$ after $S_{r+1}$ but not before.  If $I(z)$ lies to the left of $I(b_k)$ before $S_{r+1}$, then $z$ peeks into $a_kb_k$  and would have participated in $S_{r+1}$, a contradiction.  Therefore, $R(b_k) < L(z) < M$.  By Lemma~\ref{xd-lemma},  we know $z = d_{k-1}$.  By Definition~\ref{ntm-def}, we know that either all of $d_{k-1}, d_{k-2}, \ldots, d_0$ exist and are NTM-$d$'s, or there exists an index $m$ with $0 < m < k$ for which $d_{k-1}, d_{k-2}, \ldots, d_m$ exist 
  and   are all NTM-$d$'s but $d_{m-1}$ does not exist.    Lemma~\ref{D-lemma} tells us that the endpoints of each of these NTM-$d$'s are white at the beginning of $S_{r+1}$.  In either case, in step [k.5] of Algorithm~\ref{sweep}, these NTM-$d$'s will be moved, including the interval for $z = d_{k-1}$, and thus $z$ would have participated in in $S_{r+1}$, a contradiction.

  Finally, consider $w = d_k$. In this case, $L(d_k)$ and $R(d_k)$ are moved to the right during $S_{r+1}$.  If $I(z)$ intersects $I(d_k)$ before $S_{r+1}$ but not after then $I(z)$ intersects $I(d_k)$ on the left before $S_{r+1}$.   Thus $I(z)$ intersects $I(d_k)$ but not $I(a_k)$, contradicting Definition~\ref{pairs-def}.  So $I(z)$ intersects $I(d_k)$ after $S_{r+1}$ but not before, and then at the beginning of $S_{r+1}$ we know $I(z)$ intersects $I(a_k)$ but not $I(d_k)$.  By Definition~\ref{pairs-def}, $z = a_{k-1}$ or $z = d_{k-1}$.  But $z$ does not participate in sweep $S_{r+1}$ so $z = d_{k-1}$, and since $L(z) < M=L(a_{k-1})$
    the algorithm would not redefine $I(d_k)$ in step $[k.5]$. Since $w=d_k$ participates in $S_{r+1}$, this is a contradiction.
  
  This completes the proof of (4) and of Theorem~\ref{alg-correct}.
	    $\qed$

\smallskip

Finally, in Corollary~\ref{finish-proof} we use the results in Theorem~\ref{alg-correct} to complete the proof of Theorem~\ref{big-thm}, our main result.

     \begin{corollary}
     \label{finish-proof}
     If $G$ is a twin-free interval graph with no induced graph from the forbidden set ${\cal F}$ then $G$ is a strict mixed interval graph.
     \end{corollary}
     
     \begin{proof}
     Using Proposition~\ref{peeker-prop}, fix a closed interval representation of $G$ satisfying Hypothesis~\ref{hyp}.  Color all endpoints white and apply Algorithm~\ref{sweep} to this representation.  Each sweep  $S_i$ of the algorithm starts with a strict  inclusion and the endpoints of these two intervals are white by Theorem~\ref{alg-correct}(1a).  At the end of $S_i$ these four endpoints (and others) have turned red.  By Theorem~\ref{alg-correct}(2), red endpoints never participate in sweeps, thus the algorithm must terminate when no strict inclusions remain.  At this stage we have a strict mixed interval representation, and by Theorem~\ref{alg-correct}(4) the graph represented is the original graph $G$.  Thus $G$ is a strict mixed interval graph.  \qed
     \end{proof}
     
      \subsection{Complexity}
      
      Given a twin-free, $\cal F$-free graph $G = (V, E)$ we can determine if $G$ is an interval graph in $O(|V| + |E|)$ time \cite{BoLeu}. If $G$ is an interval graph, by Proposition~\ref{peeker-prop} we can obtain a representation of $G$ satisfying Hypothesis~\ref{hyp}. Now sort the endpoints of the intervals so they are listed in increasing order.  Associate with each endpoint a data structure containing the name of  the vertex to which it belongs, its color, whether it is a right or left endpoint and whether the endpoint is open or closed.  Depending on the representation, this requires no more than $O(|V| \log |V|)$ time.  Sweep through the list, enqueuing each vertex when its left endpoint is encountered and dequeuing the vertex when its right endpoint is encountered.  Two intervals that are not dequeued in the same order as they are enqueued represent an $ab$-pair.  Start the base step for the first sweep of Algorithm~\ref{sweep} with the  first such pair found.  It is straightforward to verify that the participants of this sweep can be identified and their endpoints modified according to Tables 1 and 2 in $O(|V|)$ time.   Since each sweep reduces the number of strict inclusions by at least one, there are at most  $O(|V|)$ sweeps and Algorithm~\ref{sweep} runs in $O(|V|^2)$ time.


\begin{thebibliography}{1}

\bibitem{BoWe99}
K.P. Bogart and D.B. West.
\newblock A short proof that ``proper = unit".
\newblock {\em Discrete Math}, 201 (1999),  21--23.

\bibitem{BoLeu}
K.S. Booth and S. Leuker.
\newblock Testing for the consecutive ones property,
interval graphs, and graph planarity using $PQ$-tree algorithms.
\newblock{\em J. Comput. Syst.
Sci.}, 13 (1976), 335-379.

\bibitem{DoLe}
M.~Dourado, V.~Le, F.~Protti, D.~Rautenbach and J.~Szwarcfiter.
\newblock Mixed Unit Interval Graphs.
\newblock {\em Discrete Math.}, 312 (2012), 3357--3363.

\bibitem{Go80}
M.C. Golumbic.
\newblock {\em Algorithmic Graph Theory and Perfect Graphs}.
\newblock Academic  Press, New York, 1980.

\bibitem{GoTr04}
M.C. Golumbic and A.N. Trenk.
\newblock {\em Tolerance Graphs}.
\newblock Cambridge University Press, Cambridge, 2004.

\bibitem{Jo13}
F.~Joos.
\newblock A Characterization of Mixed Unit Interval Graphs.
\newblock Preprint, 2013.

\bibitem{LeRa}
V.~Le and D.~Rautenbach.
\newblock Integral Mixed Unit Interval Graphs.
\newblock {\em Discrete Applied Math.}, 161, 1028--1036 (2013)

\bibitem{RaSz13}
D.~Rautenbach and J.L.~Szwarcfiter.
\newblock Unit Interval Graphs of Open and Closed Intervals.
\newblock {\em Journal of Graph Theory}, 72 (2013), 418--429.

 


\bibitem{Ro69}
F. Roberts.
\newblock Indifference Graphs.
\newblock {\em Proof Techniques in Graph Theory},  (F. Harary, Ed.)  Academic Press  (1969), 139--146.




\end{thebibliography}
\end{document}